\documentclass[11pt,reqno]{amsart}

\usepackage{amsmath,amssymb,mathrsfs,amsthm}
\usepackage{graphicx,cite,cases}
\usepackage{xcolor}
\newtheorem{theorem}{Theorem}[section]

\newtheorem{lemma}[theorem]{Lemma}

\numberwithin{equation}{section}

\begin{document}

\title[electromagnetic scattering by rectangular cavities]{A highly efficient and accurate numerical method for the electromagnetic scattering problem with rectangular cavities}

\author{Peijun Li}
\address{Department of Mathematics, Purdue University, West Lafayette, Indiana
47907, USA.}
\email{lipeijun@math.purdue.edu}

\author{Xiaokai Yuan}
\address{School of Mathematics, Jilin University, Changchun 130012, Jilin, China}
\email{yuanxk@jlu.edu.cn}

\thanks{The first author is supported in part by the NSF grant DMS-2208256. The second author is partially supported by the NSFC grants 12201245 and 12171017.}



\begin{abstract}
This paper presents a robust numerical solution to the electromagnetic scattering problem involving multiple multi-layered cavities in both transverse magnetic and electric polarizations. A transparent boundary condition is introduced at the open aperture of the cavity to transform the problem from an unbounded domain into that of bounded cavities. By employing Fourier series expansion of the solution, we reduce the original boundary value problem to a two-point boundary value problem, represented as an ordinary differential equation for the Fourier coefficients. The analytical derivation of the connection formula for the solution enables us to construct a small-scale system that includes solely the Fourier coefficients on the aperture, streamlining the solving process. Furthermore, we propose accurate numerical quadrature formulas designed to efficiently handle the weakly singular integrals that arise in the transparent boundary conditions. To demonstrate the effectiveness and versatility of our proposed method, a series of numerical experiments are conducted.    
\end{abstract}

\maketitle

\section{Introduction}\label{Section:Introduction}

Electromagnetic cavity scattering problems find significant applications in various fields. For instance, in radar and remote sensing, a comprehensive understanding of cavity scattering is crucial for radar systems utilized in target detection, identification, and tracking \cite{BL-SICO-2014}. In wireless communication systems, cavities can arise from surrounding structures or obstacles, and analyzing cavity scattering aids in predicting signal propagation, interference, and overall system performance \cite{JV-IEEE-1991}. Furthermore, cavity scattering plays a significant role in the behavior of metamaterials and photonic devices, influencing their unique electromagnetic properties and guiding applications in areas like superlensing, cloaking, and wave manipulation \cite{GLY-Springer, PQBR-PRL-2006}. Rigorous analysis and accurate computation of cavity scattering are vital for technological progress and effectively addressing real-world challenges.

The crucial industrial and military applications of cavity scattering problems have made them a focal point of interest for both engineering and mathematical communities. In the engineering community, researchers have initiated the investigation of electromagnetic scattering by cavities filled with penetrable materials \cite{Jin-1998, LJ-IEEE-2000}. The well-posedness of cavity scattering problems has been rigorously analyzed using integral equation methods or variational approaches, with detailed studies available in \cite{ABW-MMAS-2000, ABW-JJIA-2001, LW-JCP-2013, VW-IEEE-2003}. For rectangular shaped cavity scattering, a refined stability estimate with an explicit wavenumber dependence has been derived \cite{BY-ARMA-2016, BYZ-SJMA-2012}. Lately, the exploration of subwavelength enhancement has emerged as an important theme in mathematical research \cite{LR-SJAM-2015, LZ-SIAM-2017}. In \cite{BBT-MMS-2010, BT-MMAS-2010, GLY-Springer}, the field enhancement is explored for both single and double rectangle cavities under various boundary conditions. These studies aim to provide a more profound mathematical understanding of the subwavelength enhancement phenomenon. We refer to \cite{L-JCM-2018} for a survey of recent developments in mathematical modeling and analysis of cavity scattering problems.

Numerical methods have been extensively studied for the solution of electromagnetic cavity scattering problems. Since the problem is formulated in an unbounded domain, it is essential to employ an appropriate artificial boundary condition to reformulate it into a bounded domain. Several approaches for introducing artificial boundary conditions include using the Green's function method on the aperture \cite{D-JCP-2013, LMS-SINUM-2013}, employing a perfectly matched layer in the exterior of the cavity \cite{CLX-CCP-2021}, applying the Fourier transform on the ground \cite{ABW-JJIA-2001}, constructing transparent boundary conditions by utilizing the Fourier series expansion on the semi-circle over the cavity \cite{YBL-CSIAM-2020}, or adopting the method of boundary integral equations \cite{LAG-SISC-2014}. In \cite{WDS-NMTMA-2008, WWLS-ANM-2007, BS-SISC-2005}, numerical quadrature formulas were developed to discretize the hypersingular integrals in the method of Green's function. The wave field inside the cavity was approximated using a second-order finite difference scheme. The approach was further extended to a fourth-order scheme in \cite{ZQT-2011-JCM}. In \cite{CLX-CCP-2021} and \cite{YBL-CSIAM-2020}, an adaptive finite element method was developed, combining perfectly matched layer and transparent boundary conditions to handle the possible singularity of the solution, respectively. For problems involving scattering by multiple cavities, the Gauss--Seidel technique or a preconditioned iterative method was employed to accelerate computation \cite{D-JCP-2011, LW-JCP-2013, WZ-CCP-2016, ZZ-JCP-2019}. Notably, in the methods mentioned above, the model equation is discretized within the entire cavity. When the cavity possesses a rectangular shape, the field inside can be approximated using its Fourier series expansion, which reduces the scattering problem to one-dimensional ordinary differential equations for the Fourier coefficients. In \cite{CJLL-JCP-2021}, these ordinary differential equations are discretized using a second-order finite difference scheme. The Fourier coefficients inside the cavity are expressed in terms of the Fourier coefficients on the aperture through Gaussian elimination. Finally, a linear system on the aperture is obtained by applying a finite difference scheme, which provides an effective approach to solve the scattering problem for rectangular cavities. 

This paper presents a highly efficient and accurate numerical method for solving the electromagnetic scattering problem involving multiple multi-layered rectangular cavities. We assume that these cavities are embedded in the ground, with their apertures aligned with the ground, and their interiors filled possibly with multiple layered media. To tackle this problem, we follow the approach presented in \cite{CJLL-JCP-2021}, where the field inside the cavity is expanded using its Fourier series, and the Helmholtz equation is reduced to ordinary differential equations. By studying the equations and the transmission conditions across each layered medium, we establish a connection formula that links the Fourier coefficients of the solution in each layer to the Fourier coefficients of the solution on the aperture. This connection formula, along with the transparent boundary condition on the aperture, leads to a small $N$-by-$N$ linear system, where $N$ represents the Fourier truncation number. To efficiently generate the linear system, we design an alternative transparent boundary condition, which only involves weakly singular integrals. However, due to the singularity of the Hankel function, direct application of high-order quadrature formulas is not feasible. To address this issue, we utilize the power series of Bessel functions to deduce a recursive formula, enhancing the regularity of the integrand function. This enables us to adopt high-order Gaussian quadratures. Once the system is solved, the field inside the cavity can be obtained immediately using the connection formula.

The proposed method offers the advantage of significantly reducing memory and computational costs. As we only need to solve the system on the aperture and store the connection formula of the Fourier coefficients, the required computational resources are dramatically reduced. A series of numerical experiments is conducted to demonstrate the efficiency and versatility of our proposed method. It proves to be efficient and accurate in handling the cavity scattering problem in both transverse magnetic (TM) and transverse electric (TE) polarizations.

The structure of this paper is as follows. Section \ref{Section:Formula} focuses on the model formulation, wherein the two fundamental polarizations are introduced. In Sections \ref{Section:TM} and \ref{Section:TE}, we derive the connection formula for scattering in TM and TE polarization, respectively. This includes scenarios with a single empty cavity, a single multi-layered cavity, and multiple cavities filled with multi-layered media. Section \ref{Section:DtN} is dedicated to deriving an alternative artificial boundary condition and proposing quadrature formulas for the involved weakly singular integrals. In Section \ref{Section:NE}, we provide numerical examples to demonstrate the features of the proposed method. The paper concludes with overall reflections and avenues for future research in Section \ref{Section:C}.

\section{Problem formulation}\label{Section:Formula}

We examine the electromagnetic scattering by rectangular cavities situated within an unbounded ground plane. Given the time dependence of the electromagnetic field as $e^{-{\rm i}\omega t}$, with $\omega>0$ denoting the angular frequency, the wave propagation obeys the time-harmonic Maxwell's equations:
\begin{equation}\label{eh}
 \nabla\times\boldsymbol E={\rm i}\omega\mu\boldsymbol H,\quad
\nabla\times\boldsymbol H=-{\rm i}\omega\epsilon\boldsymbol E
+\sigma\boldsymbol E,
\end{equation}
where $\boldsymbol E$ and $\boldsymbol H$ represent the electric field and the magnetic field, respectively, $\mu$ is the magnetic permeability, $\epsilon$ denotes the electric permittivity, and $\sigma$ stands for the electrical conductivity. We assume that the medium is non-magnetic, implying a constant magnetic permeability $\mu$ throughout. However, the electric permittivity $\epsilon$ and electrical conductivity $\sigma$ are allowed to vary as spatial functions.

In this work, we focus on the electromagnetic scattering problem in TM and TE polarizations.
Let $\Omega\subset\mathbb R^2$ represent the cross-section of the $z$-invariant cavity. Its boundary is denoted by $\partial\Omega=\Gamma_c\cup \Gamma_g$, where $\Gamma_c$ represents the boundary of the cavity, including the vertical walls and the horizontal bottom, while $\Gamma_g$ denotes the infinite ground plane. The aperture of the cavity, aligned with $\Gamma_g$, is denoted by $\Gamma$. The cavity may be filled vertically with a layered inhomogeneous medium. The problem geometry is illustrated in Figure \ref{ex: Geometry}.

\begin{figure}[ht]
\centering
\includegraphics[width=0.65\textwidth]{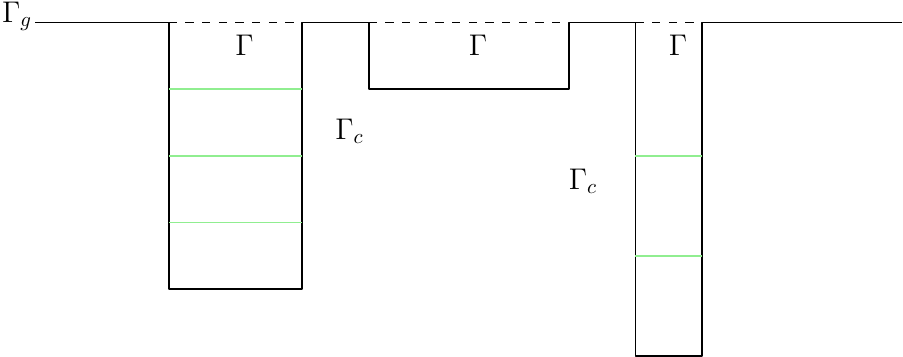}
\caption{Geometry of multiple multi-layered rectangular cavities.}
\label{ex: Geometry}
\end{figure}

 In TM polarization, the incident and total electric fields are perpendicular to the magnetic field and take the form $\boldsymbol E^i=(0, 0, u^i), \boldsymbol E=(0, 0, u)$. It can be verified from \eqref{eh} that $u$ satisfies the two-dimensional Helmholtz equation
\begin{equation}\label{HelmholtzA}
 \Delta u+\kappa^2 u=0\quad\text{in} ~ \mathbb R^2_+\cup\Omega,
\end{equation}
where the wavenumber $\kappa=(\omega^2\epsilon\mu+{\rm i}\omega\mu\sigma)^{1/2}$ with $\Im\kappa\geq 0$. For TE polarization, the magnetic field can be represented as $\boldsymbol H=(0, 0, u)$. Similarly, we can show from \eqref{eh} that $u$ satisfies the two-dimensional generalized Helmholtz equation
\begin{equation}\label{ModifiedHA}
 \nabla\cdot(\kappa^{-2}\nabla u)+u=0\quad\text{in} ~ \mathbb R^2_+\cup\Omega.
\end{equation}

By considering that both the ground plane and the cavity boundary exhibit perfect electric conductivity (PEC), we have 
\begin{equation}\label{pec-pec}
 \nu\times\boldsymbol{E}=0 \quad\text{on}~ \Gamma_g\cup\Gamma_c,
\end{equation}
where $\nu$ denotes the unit normal vector to the surfaces $\Gamma_g$ and $\Gamma_c$. Under TM polarization, the boundary condition \eqref{pec-pec} simplifies to
\begin{equation}\label{pec-pec-TM}
u=0 \quad\text{on}~ \Gamma_g\cup\Gamma_c.
\end{equation}
In TE polarization, the boundary condition \eqref{pec-pec} is equivalent to 
\begin{equation}\label{pec-pec-TE}
\partial_\nu u=0 \quad\text{on}~ \Gamma_g\cup\Gamma_c.
\end{equation}

When the medium in the upper half space is homogeneous and isotropic, it can be characterized by a constant wavenumber denoted as $\kappa_0$. Let the cavity be illuminated from above by a time-harmonic plane wave 
\[
u^i(x, y)=e^{{\rm i}(\alpha x-\beta y)}, \quad \alpha=\kappa_0\sin\theta, 
\quad \beta=\kappa_0\cos\theta,
\]
where $\theta\in (-\pi/2, \pi/2)$ is the incident angle. The scattered field $u^s$ in TM polarization can be described as $u^s=u-u^i+u^r$, and in TE polarization, it is given by $u^s=u-u^i-u^r$, where $u^r=e^{{\rm i}(\alpha x+\beta y)}$ is referred to as the reflection field. In both cases, the scattered field satisfies the Sommerfeld radiation condition:
\begin{equation}\label{Sommerfeld}
	\lim\limits_{r\rightarrow\infty} \sqrt{r}\left(\partial_r u^s-{\rm i}\kappa_0 u^s\right)=0,
	\quad r=\sqrt{x^2+y^2}.
\end{equation}

Based on \eqref{Sommerfeld}, the transparent boundary conditions (TBC) can be formulated on the aperture \cite{ZQT-2011-JCM}.  In TM polarization, the TBC can be expressed as  
\begin{equation}\label{TBC-TM}
	\partial_{y} u=I_{\rm TM}(u)+f,\quad f=-2{\rm i}\beta e^{{\rm i}\alpha x},
\end{equation}
where $I_{\rm TM}$ represents the Dirichlet-to-Neumann (DtN) operator, while the associated hypersingular integral is defined in the sense of the Hadamard finite-part, and its expression is given by 
\begin{equation}\label{Hadamard}
	I_{\rm TM}(u)(x)=\frac{{\rm i}\kappa_0}{2}\int_{\Gamma} \frac{1}{|x-x'|}
	H_1^{(1)}(\kappa_0 |x-x'|)u(x', 0){\rm d}x'. 
\end{equation}
Here, $H_1^{(1)}$ represents the Hankel function of the first kind with order one. In TE polarization, the TBC is described by
\begin{equation}\label{TBC-TE}
	u=I_{\rm TE}(u)+g, \quad g=2e^{{\rm i}\alpha x},
\end{equation}
where the Neumann-to-Dirichlet (NtD) operator $I_{\rm TE}$ involves a weakly singular integral and is defined by
\begin{equation}\label{NtD}
I_{\rm TE}(u)(x)=-\frac{{\rm i}}{2}\int_{\Gamma}  H_0^{(1)}\left(\kappa_0|x-x'|\right)\partial_{y'} u(x', 0){\rm d}x'.
\end{equation}
Here, $H_0^{(1)}$ is the Hankel function of the first kind with order zero. 

It is important to note that the quantities $u$ and $\partial_y u$ appearing in \eqref{TBC-TM} and \eqref{TBC-TE} are to be interpreted as being evaluated from the upper side of $\Gamma$, indicated more precisely as $u(x, 0^+)$ and $\partial_y u(x,0^+)$, respectively.

\section{TM polarization}\label{Section:TM}

In this section, we focus on the boundary value problem \eqref{HelmholtzA}, \eqref{pec-pec-TM}, \eqref{TBC-TM} in TM polarization. To present our findings clearly, we start by considering the scattering phenomena from a single empty cavity. Subsequently, we extend our investigation to the case of a single multi-layered cavity. Finally, we explore the general scenario of multiple cavities, with each cavity being filled possibly with a multi-layered medium.

\subsection{A single empty cavity}\label{Sub1TM}

Assuming that the empty cavity is filled with a homogeneous medium characterized by the wavenumber $\kappa=\kappa_0$. Let $\Gamma=[a, b]\times\{0\}$ denote the aperture of the cavity, where $a$ and $b$ are the coordinates along the $x$-axis, and the width and depth of the cavity are represented by $w=b-a$ and $h$, respectively.

Because of the boundary condition \eqref{pec-pec-TM}, the total field inside the cavity can be approximated using sine functions, i.e., it admits the Fourier series expansion
\begin{equation}\label{fse}
u(x, y)=\sum\limits_{n=1}^N u^{(n)}(y)\sin\frac{n\pi (x-a)}{w},
\end{equation}
where $N$ is a positive integer that controls the accuracy of the numerical solution. Substituting \eqref{fse} into \eqref{HelmholtzA}, we obtain the second order ordinary differential equations for the Fourier coefficients
\begin{eqnarray*}
 u^{(n)''}(y)+\beta_n^2 u^{(n)}(y)=0,\quad y\in(-h, 0), 
\end{eqnarray*}
where $\beta_n=\big(\kappa^2-\left(\frac{n\pi}{w}\right)^2\big)^{1/2}, \Im \beta_n\geq 0$ for $n=1, \dots, N.$

Let $u^{(n)}, n=1, \dots, N$, be the Fourier coefficients of the solution on $\Gamma$. Since $u$ vanishes on the bottom of the cavity, we have from \eqref{fse} that $u^{(n)}(-h)=0, n=1, \dots, N$. Therefore, the $n$-th order Fourier coefficient satisfies a two-point boundary value problem
\begin{equation}\label{singleTM-ODEs}
\left\{
\begin{aligned}
&u^{(n)''}(y)+\beta_n^2 u^{(n)}(y)=0, \quad -h< y< 0,\\
&u^{(n)}(0)=u^{(n)},\quad u^{(n)}(-h)=0.
\end{aligned}
\right.
\end{equation}
A straightforward calculation shows that the solution to \eqref{singleTM-ODEs} is given by 
\begin{equation}\label{SingleTM-un}
u^{(n)}(y)=\left\{
\begin{aligned}
&\frac{1}{1-e^{2{\rm i}\beta_n h}}\big(e^{-{\rm i}\beta_n y}-e^{2{\rm i}\beta_n h}e^{{\rm i}\beta_n y}\big)u^{(n)} &\quad\text{if } ~ \beta_n\neq 0,\\
&\Big(1+\frac{1}{h} y\Big)u^{(n)} &\quad \text{if } ~\beta_n=0. 
\end{aligned}
\right.
\end{equation}

Let $n_0$ denote the term that could lead to $\beta_{n}=0$. Combining \eqref{fse} and \eqref{SingleTM-un} yields that the expression for $u(x, y)$ is given as follows:
\begin{align}\label{singleTM-u}
	u(x, y)&= \sum\limits_{n=1, n\neq n_0}^N \frac{1}{1-e^{2{\rm i}\beta_n h}}\big(e^{-{\rm i}\beta_n y}
	-e^{2{\rm i}\beta_n h}e^{{\rm i}\beta_n y}\big)\sin\frac{n\pi (x-a)}{w}u^{(n)}\notag\\
	&\quad +\Big(1+\frac{1}{h} y\Big)\sin\frac{n_0 \pi (x-a)}{w}u^{(n_0)}.
\end{align}
Taking the normal derivative of \eqref{singleTM-u} from the lower side of $\Gamma$, we obtain 
\begin{align}\label{singleTM-du}
	\partial_y u(x, 0^-)=\sum\limits_{n=1}^N s^{(n)}u^{(n)}\sin\frac{n\pi  (x-a)}{w},
\end{align}
where $s^{(n_0)}=1/h$ and $s^{(n)}=-{\rm i}\beta_n(1+e^{2{\rm i}\beta_n h})/(1-e^{2{\rm i}\beta_n h})$ for $n\neq n_0$. 
Using the transmission conditions on $\Gamma$: $u(x, 0^+)=u(x, 0^-), \partial_y u(x, 0^+)=\partial_y u(x, 0^-),$ and substituting \eqref{singleTM-du} into the TBC \eqref{TBC-TM} leads to  
\begin{eqnarray}\label{singleTM-tbc}
	\sum\limits_{n=1}^N s^{(n)}u^{(n)}\sin\frac{n\pi  (x-a)}{w}
	=\sum\limits_{n=1}^N u^{(n)} I_{\rm TM}\Big(\sin\frac{n\pi (x-a)}{w}\Big)+f.
\end{eqnarray}

Multiplying both sides of \eqref{singleTM-tbc} by $\sin\frac{m \pi (x-a)}{w}$, where $m=1, ...,N$,  and integrating over $\Gamma$, we obtain 
\begin{equation}\label{SingleTM}
	D_{\rm TM}U=M_{\rm TM}U+F,
\end{equation}
where $U=[u^{(1)}, u^{(2)}, \dots, u^{(N)}]^\top$, the $m$-th entry of the vector $F$ is defined by 
\begin{equation*}
	F(m)=-2{\rm i}\beta\int_{0}^{w} e^{{\rm i}\alpha (x+a)}\sin\frac{m\pi  x}{w}{\rm d}x,
\end{equation*}
$D_{\rm TM}$ is a diagonal matrix with its diagonal entries  $D_{\rm TM}(m, m)=w s^{(m)}/2,$
and the elements of the matrix $M_{\rm TM}$, with dimensions $N\times N$, are expressed as 
\begin{equation}\label{TM-SingleDtNOld}
M_{\rm TM}(m, n)=\int_{\Gamma} \sin\frac{m\pi  (x-a)}{w} I_{\rm TM}\Big(\sin\frac{n\pi (x'-a)}{w}\Big){\rm d}x.
\end{equation}

Once the linear system \eqref{SingleTM} is solved, the solution inside the cavity can be explicitly computed using the expansion \eqref{singleTM-u}. However, the matrix entries \eqref{TM-SingleDtNOld} contain hypersingular integrals, which make it challenging to devise a high-order quadrature. In Section \ref{Section:DtN}, we will reformulate \eqref{TBC-TM} to an alternative form that incorporates weakly singular integrals. Additionally, we will introduce high-order quadrature formulas to effectively tackle this concern.

\subsection{A single multi-layered cavity}\label{Sub2TM}

Consider a cavity situated at $[a, b]$ and filled with an $L$-layered inhomogeneous medium, where $L\geq 2$ is an integer. Let $w=b-a$ and $h$ be the width and depth of the cavity, respetively. Denote the sequence of values $-h=y_L<y_{L-1}<\cdots<y_1<y_0=0$. Consider the interfaces between different media represented by $\Gamma_l=[a, b]\times\left\{y_l\right\}$, where $l=1,...,L-1$. The aperture and bottom of the cavity are denoted by $\Gamma_0$ and $\Gamma_L$, respectively. For each $l=1, ..., L$, let $\Omega_l$ be the region between $\Gamma_{l-1}$ and $\Gamma_{l}$. It is assumed that the medium in each layer is homogeneous and characterized by a constant wavenumber $\kappa_l$.

Let $u_l(x, y)$ denote the total field within the $l$-th layer, which satisfies 
\begin{equation}\label{TM-Hui}
\Delta u_l(x, y)+\kappa_l^2 u_l(x, y)=0\quad \text{in } ~ \Omega_l,
\end{equation} 
along with the transmission conditions on the interface $\Gamma_l$:
\begin{equation}\label{TM-tci}
	u_l(x, y_l)=u_{l+1}(x, y_l),\quad \partial_y u_l(x, y_l)=\partial_y u_{l+1}(x, y_l).
\end{equation}
The solution of \eqref{TM-Hui} can be approximated by 
\begin{equation}\label{TM-ui}
u_l(x, y)=\sum\limits_{n=1}^{N}u_l^{(n)}(y)\sin\frac{n\pi (x-a)}{w}.
\end{equation}

Consider $u_l^{(n)}, n=1, \dots, N$, for $l=0, 1, \dots, L$, as the Fourier coefficients of the solution on $\Gamma_l$. 
Note from \eqref{pec-pec-TM} that $u_L^{(n)}=0,  n=1, \dots, N$. For $l=1, \dots, L$, substituting \eqref{TM-ui} into \eqref{TM-Hui} reveals that the Fourier coefficients satisfy 
\begin{equation}\label{MultiTM-ODE}
\left\{
\begin{aligned}
&u_l^{(n)''}(y)+\big(\beta_l^{(n)}\big)^2 u_l^{(n)}(y)=0,\quad y_l< y< y_{l-1},\\
& u_l^{(n)}(y_l)=u_l^{(n)},\quad u_l^{(n)}(y_{l-1})=u_{l-1}^{(n)},
\end{aligned}
\right.
\end{equation}
where $\beta_l^{(n)}=\big(\kappa_l^2-\left(\frac{n\pi}{w}\right)^2\big)^{1/2}$ with $\Im\beta_l^{(n)}\geq 0.$

A straightforward computation shows that the solution of \eqref{MultiTM-ODE} is given by
\begin{equation}\label{MultiTM-u}
u_l^{(n)}(y)=\left\{
\begin{aligned}
&\frac{1}{\zeta_l}\Big[\big(e^{{\rm i}\beta_l^{(n)}\left(y-y_{l-1}\right)}-e^{-{\rm i}\beta_l^{(n)}\left(y-y_{l-1}\right)}\big)u_l^{(n)}\\
&\hspace{1cm}-\big(e^{{\rm i}\beta_l^{(n)}\left(y-y_{l}\right)}-e^{-{\rm i}\beta_l^{(n)}\left(y-y_{l}\right)}\big)u_{l-1}^{(n)}\Big] &\quad\text{if } ~\beta_l^{(n)}\neq 0,\\
&\frac{1}{h_l}\Big[\big(u_l^{(n)}-u_{l-1}^{(n)}\big)y+ u^{(n)}_{l-1}y_l-u^{(n)}_l y_{l-1}\Big] &\quad\text{if }~\beta_l^{(n)}=0,
\end{aligned}
\right.
\end{equation}
where $h_l=y_l-y_{l-1}, \zeta_l=e^{{\rm i}\beta_l^{(n)}h_l}-e^{-{\rm i}\beta_l^{(n)}h_l}$.

Taking the derivatives of \eqref{MultiTM-u} from the upper and lower sides of $\Gamma_l$, and using \eqref{TM-tci}--\eqref{TM-ui}, we deduce a symmetric tri-diagonal linear system for cases where $L\geq 3$, which is referred to as the  connection formula and is presented as 
\begin{equation}\label{MultiTM-V}
D_{\rm TM}^{(n)}\boldsymbol{u}^{(n)}=\boldsymbol{b}_{\rm TM}^{(n)},
\end{equation}
where $\boldsymbol u^{(n)}=[u_1^{(n)}, u_2^{(n)}, \dots, u_{L-1}^{(n)}]^\top$, $\boldsymbol{b}_{\rm TM}^{(n)}=[-a_1^{(n)} u_0^{(n)}, 0,\cdots, 0]^\top$, and the diagonal and sub-diagonal entries of the matrix $D_{\rm TM}^{(n)}$ are given by
\[
D_{\rm TM}^{(n)}(l, l)=b_l^{(n)}+b_{l+1}^{(n)},\quad D_{\rm TM}^{(n)}(l, l+1)=a_{l+1}^{(n)}.
\]
Here
\begin{align}
\label{CoeffTM-Single-a}a_l^{(n)}&=\left\{
\begin{aligned}
&-\frac{1}{h_l}\quad &\text{if } ~\beta_l^{(n)}=0,\\
&-2\frac{{\rm i}\beta_l^{(n)}}{\zeta_l} \quad &\text{if } ~ \beta_l^{(n)}\neq 0,
\end{aligned}
\right.\\
\label{CoeffTM-Single-b}b_l^{(n)}&=\left\{
\begin{aligned}
&\frac{1}{h_l}\quad &\text{if } ~\beta_l^{(n)}=0,\\
&\frac{{\rm i}\beta_l^{(n)}}{\zeta_l}\big(e^{{\rm i}\beta_l^{(n)}h_l}+e^{-{\rm i}\beta_l^{(n)}h_l}\big) \quad  &\text{if }~\beta_l^{(n)}\neq 0.
\end{aligned}
\right.
\end{align}
When $L=2$, the connection formula \eqref{MultiTM-V} is simply reduced to
\begin{equation}\label{MultiTM-V1}
\big(b_1^{(n)}+b_2^{(n)}\big)u_1^{(n)}=-a_1^{(n)} u_0^{(n)}.
\end{equation}

By solving \eqref{MultiTM-V} or \eqref{MultiTM-V1}, we can retrieve the Fourier coefficients of the solution on the interfaces ${\boldsymbol u}^{(n)}$ in terms of the Fourier coefficients of the solution on the aperture $u_0^{(n)}$. Since $u_0^{(n)}$ is unknown, in practice, we can first solve \eqref{MultiTM-V} or \eqref{MultiTM-V1} with the right-hand side $\hat{\boldsymbol b}^{(n)}=[1, 0,.., 0]^\top$, and denote the corresponding solution as $\hat{\boldsymbol u}^{(n)}=[\hat{u}_1^{(n)}, \hat{u}_2^{(n)},\dots, \hat{u}_{L-1}^{(n)}]^\top$. Then, the solution ${\boldsymbol u}^{(n)}$  can be obtained by multiplying the factor $-a_1^{(n)} u_0^{(n)}$ to $\hat{\boldsymbol u}^{(n)}$.

Next, we discuss the approach to determining $u_0^{(n)}, n=1, \dots, N$. For the case when $l=1$, we take the normal derivative of \eqref{TM-ui} on $\Gamma$ from below, leading to  
\begin{equation}\label{MultiTM-du0}
\partial_y u_1(x, 0^-)=\sum\limits_{n=1}^N {\hat s}^{(n)}u_0^{(n)}\sin\frac{n \pi (x-a)}{w},
\end{equation}
where ${\hat s}^{(n)}=-b_1^{(n)}+(a_1^{(n)})^2\hat{u}_1^{(n)}$. We mention that the normal derivative can also be described in the form of \eqref{MultiTM-du0} with ${\hat s}^{(n)}=-b_1^{(n)}$ for the case where $L=1$. 

Utilizing the transmission conditions on $\Gamma$: $u(x, 0^+)=u_1(x, 0^-), \partial_y u(x, 0^+)=\partial_y u_1(x, 0^-)$, and replacing \eqref{singleTM-du} with \eqref{MultiTM-du0} and following the same discussion as that in Subsection \ref{Sub1TM}, we arrive at an analogous formulation of the system depicted in \eqref{SingleTM}: $D_{\rm TM}U=M_{\rm TM}U+F$, except that the diagonal matrix $D_{\rm TM}$ is now defined as $D_{\rm TM}(m, m)=w{\hat s}^{(m)}/2, m=1, \dots, N.$. After the system is solved, the solution in each layer can be computed explicitly by \eqref{TM-ui}.

\subsection{Multiple multi-layered cavities}\label{Sub3TM}

Assume that the ground contains a total of $K$ cavities, and we denote the aperture of the $k$-th cavity as $\Gamma_k=[a_k, b_k]\times\left\{0\right\}$, with a width of $w_k=b_k-a_k$. Inside the $k$-th cavity, it is filled with an inhomogeneous medium consisting of $L_k$ layers. The interfaces between different layers are denoted by $\Gamma_{k, l}=[a_k, b_k]\times\left\{y_l^k\right\}$, where $l=1, .., L_{k}-1$. The aperture and the bottom of the $k$-th cavity are denoted by $\Gamma_{k, 0}$ and $\Gamma_{k, L_k}$, respectively. In the $l$-th layer of the $k$-th cavity, we assume that the medium is homogeneous and characterized by a constant wavenumber $\kappa_{k, l}$.

Denote by $u(x, y; k)$ the total field restricted in the $k$-th cavity. The DtN operator \eqref{Hadamard} can be written as
\[
I_{\rm TM}(u)(x)=\frac{{\rm i}\kappa_0}{2}\sum\limits_{k=1}^K\int_{\Gamma_k} \frac{1}{|x-x'|}H_1^{(1)}(\kappa_0 |x-x'|)u(x', 0; k){\rm d}x'.
\]
The solution in the $l$-th layer of the $k$-th cavity can be approximated by  
\begin{equation}\label{MulMulTM-u}
u_l(x, y; k)=\sum\limits_{n=1}^N u_l^{(n)}(y; k)\sin\frac{n\pi(x-a_k)}{w_k},\quad y\in(y_{l}^k, y_{l-1}^k).
\end{equation}
It can be verified that the Fourier coefficients of \eqref{MulMulTM-u} satisfy the same system as \eqref{MultiTM-ODE}, and the solution is given by
\begin{equation}\label{MultiMulTM-u}
u_l^{(n)}(y; k)=\left\{
\begin{aligned}
&\frac{1}{\zeta^k_l}\Big[\big(e^{{\rm i}\beta_{l,k}^{(n)}(y-y^k_{l-1})}
	-e^{-{\rm i}\beta_{l,k}^{(n)}(y-y^k_{l-1})}\big)u_{l,k}^{(n)}\\
	&\hspace{1cm}
	-\big(e^{{\rm i}\beta_{l, k}^{(n)}(y-y^k_{l})}-e^{-{\rm i}\beta_{l, k}^{(n)}(y-y^k_{l})}\big)u_{l-1,k}^{(n)}\Big]
	&\quad\text{if } ~\beta_{l,k}^{(n)}\neq0,\\
&	\frac{1}{h_l^k}\Big[\big(u_{l, k}^{(n)}-u_{l-1, k}^{(n)}\big)y+y^k_l u^{(n)}_{l-1, k}-y^k_{l-1}u^{(n)}_{l, k}\Big]
	&\quad\text{if } ~\beta_{l,k}^{(n)}=0,
\end{aligned}
\right.
\end{equation}
where $u^{(n)}_{l, k}, n=1, \dots, N, l=0, 1, \dots, L, k=1, \dots, K$ are the Fourier coefficients of the solution on $\Gamma_{k, l}$, $h^k_l=y^k_l-y^k_{l-1}, \zeta^k_l=e^{{\rm i}\beta_{l,k}^{(n)}h^k_l}-e^{-{\rm i}\beta_{l,k}^{(n)}h^k_l}$, and 
$\beta_{l,k}^{(n)}=\big(\kappa_{k, l}^2-\big(\frac{n\pi}{w_k}\big)^2\big)^{1/2}, \Im\beta_{l,k}^{(n)}\geq 0.$

Continuing in a manner analogous to the discourse presented in Subsection \ref{Sub2TM}, we can deduce an equivalent connection formula, similar to the one provided in \eqref{MultiTM-V}, with the inclusion of an additional subscript $k$. For $l=1$, the normal derivative of \eqref{MulMulTM-u} on the aperture of the $k$-th cavity is given by
\begin{equation}\label{MultiMulTM-du0}
\partial_y u_1(x, 0^-; k)= \sum\limits_{n=1}^N s_k^{(n)}u_{0,k}^{(n)}\sin\frac{n \pi (x-a_k)}{w_k},
\end{equation}
where $s_k^{(n)}=-b_{1, k}^{(n)}+(a_{1, k}^{(n)})^2\hat{u}_{1,k}^{(n)}$, with the coefficients $a_{l, k}^{(n)}$ and $b_{l, k}^{(n)}$ defined in \eqref{CoeffTM-Single-a}--\eqref{CoeffTM-Single-b}, with an additional subscript $k$. When $L_k=1$, the coefficient is $s_k^{(n)}=-b_{1, k}^{(n)}$.

Likewise, using the transmission conditions: $u(x,0^+)=u_1(x,0^-;k), \partial_y u(x, 0^+)=\partial_y u_1(x,0^-;k), x\in(a_k, b_k)$, upon substituting \eqref{singleTM-du} with \eqref{MultiMulTM-du0}, and proceeding with a comparable discussion as presented in Subsection \ref{Sub1TM}, we can deduce the linear system governing the Fourier coefficients of the solution on the apertures:
\begin{equation}\label{MultiTM}
{\mathbb D}_{\rm TM}{\mathbb U}={\mathbb M}_{\rm TM}{\mathbb U}+{\mathbb F},
\end{equation}
where 
\[
 {\mathbb U}=\begin{bmatrix}
	U_1 \\ U_2 \\ \vdots \\ U_K
	\end{bmatrix},\quad
 {\mathbb F}=\begin{bmatrix}
	F_1 \\ F_2 \\ \vdots \\ F_K
	\end{bmatrix},\quad 
	{\mathbb D}_{\rm TM}=\begin{bmatrix}
	D^{(1)}_{\rm TM} & 0 & \cdots & 0\\
	0 & D^{(2)}_{\rm TM} & \cdots & 0\\
	 \vdots & \vdots & \ddots & \vdots\\
	0 & 0 & \cdots & D^{(K)}_{\rm TM}
	\end{bmatrix}
\]
with $U_k=[u_{0, k}^{(1)}, u_{0, k}^{(2)}, \dots, u_{0, k}^{(N)}]^\top$, 
$D^{(k)}_{\rm TM}=\frac{w_k}{2}{\rm diag}(s_k^{(1)}, \dots, s_k^{(N)})$, the elements of the $k$-th block of the vector ${\mathbb F}$ being given by 
\[
 F_k(m)=-2{\rm i}\beta\int_{0}^{w_k} e^{{\rm i}\alpha (x+a_k)}\sin\frac{m\pi x}{w_k}{\rm d}x,
\]
and the elements within the block matrix ${\mathbb M}_{\text{TM}}$ consist of hypersingular integrals, which will be transformed into forms involving weakly singular integrals. An examination of this transformation process will be presented in Section \ref{Section:DtN}.

Once the system \eqref{MultiTM} is solved, the Fourier coefficients on the inner interfaces of the $k$-th cavity can be obtained through the connection formula. The total field in the $l$-th layer of the $k$-th cavity can be computed using \eqref{MulMulTM-u}--\eqref{MultiMulTM-u}.

\section{TE polarization}\label{Section:TE}

In this section, we provide a concise overview of the numerical solution for the boundary value problem \eqref{ModifiedHA}, \eqref{pec-pec-TE}, \eqref{TBC-TE} in TE polarization, as its fundamental concept aligns closely with that of the TM polarization.

\subsection{A single empty cavity}\label{Sub1TE}

Consider the scattering by a single empty cavity. Due to the boundary condition \eqref{pec-pec-TE}, the total field can be approximated by the Fourier series of cosine functions:
\begin{equation}\label{TE-fse}
u(x, y)=\sum\limits_{n=0}^N u^{(n)}(y)\cos\frac{n\pi(x-a)}{w}.
\end{equation}

Denote by $u^{(n)}, n=0, 1, \dots, N,$ the Fourier coefficients of the solution on $\Gamma$. Substituting \eqref{TE-fse} into \eqref{ModifiedHA} and \eqref{pec-pec-TE} yields that the Fourier coefficients satisfy 
\[
\left\{
\begin{aligned}
&u^{(n)''}(y)+\beta_n^2 u^{(n)}(y)=0, \quad -h< y< 0,\\
&u^{(n)}(0)=u^{(n)},\quad u^{(n)'}(-h)=0,
\end{aligned}
\right.
\]
which has a unique solution given by  
\begin{equation}\label{SingleTE-un}
u^{(n)}(y)=\left\{
\begin{aligned}
& \frac{1}{1+e^{2{\rm i}\beta_n h}}\big(e^{-{\rm i}\beta_n y}+e^{2{\rm i}\beta_n h}e^{{\rm i}\beta_n y}\big)u^{(n)}
\quad &\text{if } ~\beta_n\neq 0,\\
& u^{(n)}\quad &\text{if } ~ \beta_n=0. 
\end{aligned}
\right.
\end{equation}

Taking the normal derivative of \eqref{TE-fse} on $\Gamma$ from below and using \eqref{SingleTE-un} gives 
\begin{equation}\label{SingleTE-du}
	\partial_y u(x, 0^-)=\sum\limits_{n=0}^N 
	t^{(n)}u^{(n)}\cos\frac{n\pi(x-a)}{w},
\end{equation}
where $t^{(n_0)}=0$ and $t^{(n)}={\rm i}\beta_n(e^{2{\rm i}\beta_n h}-1)/(1+e^{2{\rm i}\beta_n h})$ for $n\neq n_0$. 

As the cavity is assumed to be empty, and the wavenumber within the cavity remains consistent with that in free space, the following transmission conditions are imposed on $\Gamma$: $u(x, 0^+)=u(x, 0^-), \partial_y u(x, 0^+)=\partial_y u(x, 0^-)$. Substituting \eqref{SingleTE-du} into \eqref{TBC-TE} leads to 
\begin{align}\label{SingleTE-identity}
	&\sum\limits_{n=0}^N u^{(n)}(0)\cos\frac{n\pi(x-a)}{w}\\
	&= -\frac{\rm i}{2}\int_{\Gamma} H_0^{(1)}\left(\kappa|x-x'|\right)
	\bigg[\sum\limits_{n=0}^N 
	t^{(n)} u^{(n)}(0)
	\cos\frac{n\pi(x'-a)}{w}\bigg]{\rm d}x'+g(x).\notag
\end{align}
Multiplying both sides of \eqref{SingleTE-identity} with $\cos\frac{m\pi(x-a)}{w}, m=0, 1, \dots, N$, and integrating on the aperture, we obtain    
 \begin{equation}\label{SingleTE-key}
	D_{\rm TE}U=\tilde{M}_{\rm TE}U+G, \quad
	\tilde{M}_{\rm TE}(m, n)=t^{(n)} M_{\rm TE}(m, n),
\end{equation}
where $U=[u^{(0)}, u^{(1)}, \dots, u^{(N)}]^\top$, $D_{\rm TE}$ is a diagonal matrix whose diagonal entry
is $\frac{w}{2}$ unless for $D_{\rm TE}(0, 0)=w$,  and the $m$-th component of vector $G$ is
\begin{equation}\label{SingleTE-g}
G(m) = 2\int_{\Gamma} e^{{\rm i}\alpha x}\cos\frac{m\pi(x-a)}{w}{\rm d}x 
 = 2e^{{\rm i}\alpha a}\int_{0}^w e^{{\rm i}\alpha x}\cos\frac{m\pi x}{w}{\rm d}x, 
\end{equation}
and the elements of the matrix $M_{\rm TE}$, by a change of variables, are expressed by
\begin{equation}\label{SingleTE-Kmn}
M_{\rm TE}(m, n) =
		 -\frac{\rm i}{2}\left(\frac{w}{2\pi}\right)^2\int_{0}^{2\pi}\int_{0}^{2\pi} H_0^{(1)}\Big(\frac{\kappa_0 w}{2\pi}|x-x'|\Big)
		 \cos\frac{n x'}{2}
		\cos\frac{m x}{2}{\rm d}x'{\rm d}x.
\end{equation}

The entries of $M_{\rm TE}$ contain weakly singular integrals, and an efficient quadrature formula for evaluating them is proposed in Section \ref{Section:DtN}. After solving \eqref{SingleTE-key}, the total field in the cavity can be computed explicitly by using \eqref{TE-fse}--\eqref{SingleTE-un}.

\subsection{A single multi-layered cavity}\label{Sub2TE}

Denote by $u_l$ the total field in the $l$-th layer, which can be approximated by 
\begin{equation}\label{TE-ui}
u_l(x, y)=\sum\limits_{n=0}^N u_l^{(n)}(y)\cos\frac{n\pi(x-a)}{w}.
\end{equation}

A straightforward calculation shows that the Fourier coefficients satisfy exactly the same system \eqref{MultiTM-ODE} in the $l$-th layer. Moreover, the transmission conditions across the interface $\Gamma_l$ are 
\begin{equation*}
	u_l(x, y_l)=u_{l+1}(x, y_l),\quad \frac{1}{\kappa_l^2}\partial_y u_l(x, y_l)
	=\frac{1}{\kappa_{l+1}^2} \partial_y u_{l+1}(x, y_l),
	\quad l=1,\dots, L-1, 
\end{equation*}
which give, after using the orthogonality properties of the cosine functions, that
\begin{equation*}
	\frac{1}{\kappa_l^2}\big(a_l^{(n)} u^{(n)}_{l-1}+b_l^{(n)}u^{(n)}_l\big)=-
	\frac{1}{\kappa_{l+1}^2}\big(b_{l+1}^{(n)} u^{(n)}_l+a_{l+1}^{(n)} u^{(n)}_{l+1}\big),\quad l=1,\dots, L-1,
\end{equation*}
where $a^{(n)}_l, b^{(n)}_l$ for $l=1, \dots, L$ are defined in \eqref{CoeffTM-Single-a}--\eqref{CoeffTM-Single-b}. On the bottom of the cavity, we deduce from the boundary condition \eqref{pec-pec-TE} that
\begin{equation}\label{MultiTE-Bottom}
	a_{L}^{(n)} u^{(n)}_{L-1}+b_{L}^{(n)} u^{(n)}_L=0.
\end{equation}

Hence we can obtain a similar tri-diagonal system for the connection formula
\begin{equation}\label{MultiTE-V}
D_{\rm TE}^{(n)}\boldsymbol{u}^{(n)}=\boldsymbol{b}_{\rm TE}^{(n)}, 
\end{equation}
where $\boldsymbol u^{(n)}=[u_1^{(n)}, u_2^{(n)},\dots, u_{L}^{(n)}]^\top$, $\boldsymbol{b}_{\rm TE}^{(n)}=[-\frac{1}{\kappa_1^2}a_1^{(n)} u^{(n)}_0, 0, \dots, 0]^\top$, and the entries of the symmetric tri-diagonal matrix $D_{\rm TE}^{(n)}$ are specified by
\begin{equation}\label{ConnectionTE}
D_{\rm TE}^{(n)}(l, l)=\frac{b_l^{(n)}}{\kappa_l^2}+\frac{b_{l+1}^{(n)} }{\kappa_{l+1}^2},\quad D^{(n)}(l, l+1)=\frac{a_{l+1}^{(n)}}{\kappa_{l+1}^2}
\end{equation}
with $b_{L+1}^{(n)}/\kappa_{L+1}^2=0$. In the case of $L=1$, the connection formula \eqref{ConnectionTE} is reduced to \eqref{MultiTE-Bottom}.

Denote by $\hat{\boldsymbol u}^{(n)}=[\hat{u}_1^{(n)}, \hat{u}_2^{(n)},\dots, \hat{u}_{L}^{(n)}]^\top$
the solution of \eqref{MultiTE-V} with the right-hand side given by
$\hat{\boldsymbol b}^{(n)}=[1, 0,...,0]^\top,$ then we have ${\boldsymbol u}^{(n)}=-\frac{1}{\kappa_1^2}a_1^{(n)} u^{(n)}_0 \hat{\boldsymbol u}^{(n)}.$
Taking the derivative of \eqref{TE-ui} on the aperture from below for $l=1$, and using the transmission conditions $u(x, 0^+)=u_1(x, 0^-)$ and $\frac{1}{\kappa_0^2}\partial_y u(x, 0^+)=\frac{1}{\kappa_1^2}\partial_y u_1(x, 0^-)$, we obtain 
\begin{align}\label{SingleTE-Multi-Du}
	\partial_y u(x, 0^+)=\left(\frac{\kappa_0}{\kappa_1}\right)^2\partial_y u_1(x, 0^-)=\sum\limits_{n=0}^N {\hat t}^{(n)}u_0^{(n)}\cos\frac{\pi n(x-a)}{w}, 
\end{align}
where ${\hat t}^{(n)}=\big(\frac{\kappa_0}{\kappa_1}\big)^2\big[\frac{1}{\kappa_1^2}(a_1^{(n)})^2 \hat{u}^{(n)}_1-b_1^{(n)}\big]$.  Substituting \eqref{SingleTE-Multi-Du} into \eqref{TBC-TE}, multiplying both sides by $\cos\frac{m\pi(x-a)}{w}$, and integrating it on the aperture, we obtain 
\begin{equation}\label{MultiTE-key}
\hat D_{\rm TE} U=\hat M_{\rm TE} U+G,
\end{equation} 
where $U=[u^{(0)}, u^{(1)}, \dots, u^{(N)}]^\top$, the entries of $G$ are given in \eqref{SingleTE-g}, the diagonal matrix $\hat D_{\rm TE}$ and the matrix  $\hat M_{\rm TE}$ are given by 
\[
	\hat D_{\rm TE}(m, m)=\left\{
	\begin{aligned}
	&w\quad &\text{if }~ m=0,\\
	&\frac{w}{2}\quad &\text{if } ~ m\neq 0,
	\end{aligned}
	\right.\qquad
	\hat M_{\rm TE}(m, n)={\hat t}^{(n)} M_{\rm TE}(m, n)
\]
with the elements of $M_{\rm TE}$ being given in \eqref{SingleTE-Kmn}. After the system \eqref{MultiTE-key} is solved, the solution in each layer can be computed explicitly by \eqref{TE-ui}.

\subsection{Multiple multi-layered cavities}\label{Sub3TE}

Let $u(x, y; k)$ denote the total field of the $k$-th cavity. The NtD operator \eqref{TBC-TE} on the $k$-th aperture can be reformulated as 
\begin{eqnarray}\label{MulMul-TE-NtD}
I_{\rm TE}(u)(x)=-\frac{\rm i}{2}\sum\limits_{k=1}^K \int_{\Gamma_k} H_0^{(1)}(\kappa_0 |x-x'|)
\partial_{y'} u(x', 0; k){\rm d}x'.
\end{eqnarray}
The solution $u(x, y; k)$ in the $l$-th layer can be approximated by 
\begin{equation}\label{MulMul-TE-EXpan}
	u_l(x, y; k)=\sum\limits_{n=0}^N u_l^{(n)}(y; k)\cos\frac{n\pi (x-a_k)}{w_k}, 
	\quad y\in (y_{l}^k, y_{l-1}^k).
\end{equation}
Following the same discussion in Section \ref{Sub2TE}, we find the following expression within the $k$-th cavity:
 \begin{align*}
 	\partial_y u(x, 0^+; k) = \left(\frac{\kappa_0}{\kappa_{k,1}}\right)^2 \partial_y u_1(x, 0^-; k)=\sum\limits_{n=0}^N t_k^{(n)}u_{0, k}^{(n)}\cos\frac{n\pi(x-a_k)}{w_k}.
 \end{align*}
 where $t_k^{(n)}= \big(\frac{\kappa_0}{\kappa_{k,1}}\big)^2\big(-b_{1, k}^{(n)}+\frac{1}{\kappa_{k,1}^2}(a_{1, k}^{(n)})^2 \hat{u}_{1, k}\big)$. 
 
By substituting the above equation into \eqref{TBC-TE} and \eqref{MulMul-TE-NtD}, we deduce
\begin{equation}\label{MulMulTE-Sys}
 {\mathbb D}_{\rm TE}{\mathbb U}={\mathbb M}_{\rm TE}{\mathbb U}+{\mathbb G}, 
\end{equation}
where 
\[
 {\mathbb U}=\begin{bmatrix}
	U_1 \\ U_2 \\ \vdots \\ U_K
	\end{bmatrix},\quad
 {\mathbb G}=\begin{bmatrix}
	G_1 \\ G_2 \\ \vdots \\ G_K
	\end{bmatrix},\quad 
	{\mathbb D}_{\rm TE}=\begin{bmatrix}
	D^{(1)}_{\rm TE} & 0 & \cdots & 0\\
	0 & D^{(2)}_{\rm TE} & \cdots & 0\\
	 \vdots & \vdots & \ddots & \vdots\\
	0 & 0 & \cdots & D^{(K)}_{\rm TE}
	\end{bmatrix}
\]
with $U_k=[u_{0, k}^{(0)}, u_{0, k}^{(1)}, \dots, u_{0, k}^{(N)}]^\top$, $D^{(k)}_{\rm TE}=w_k$ for $k=1$ and $w_k/2$ otherwise, the elements of the $k$-th block of the vector ${\mathbb G}$ being defined as
\[
G_k(m)=2\int_0^{w_k} e^{{\rm i}\alpha (x+a_k)}\cos\frac{ m \pi x}{w_k}{\rm d}x,
\]
and the block form of the matrix ${\mathbb M}_{\rm TE}$ is 
\[
 	{\mathbb M}_{\rm TE}=\begin{bmatrix}
	M_{1, 1} & M_{1, 2} & \cdots & M_{1, K}\\
	M_{2, 1} & M_{2, 2} & \cdots & M_{2, K}\\
	 \vdots & \vdots & \ddots & \vdots\\
	M_{K, 1} & M_{K, 2} & \cdots & M_{K, K}
	\end{bmatrix}
\]
with the entries being given by 
\[
 M_{k, j}(m, n)=
-\frac{{\rm i} t^{(n)}_j}{2}
	\int_{0}^{w_k}\int_0^{w_j}H_0^{(1)}(\kappa_0 |x+a_k-x'-a_j|)
	\cos\frac{ n\pi x'}{w_j}\cos\frac{ m\pi x}{w_k}{\rm d}x'{\rm d}x. 
 \]

Once the system \eqref{MulMulTE-Sys} is solved, the Fourier coefficients on the inner interfaces of the $k$-th cavity can be obtained using the connection formula. Subsequently, the total field in the $l$-th layer of the $k$-th cavity can be expressed by \eqref{MulMul-TE-EXpan}, where the coefficients are defined according to \eqref{MultiTM-u}.
 
\section{Numerical quadratures}\label{Section:DtN}

Since the first-order Hankel function $H_1^{(1)}(t)$ is singular at $t=0$, the DtN operator \eqref{Hadamard} for the TM polarization is defined in the sense of the Hadamard finite-part. Developing a high-order numerical integral quadrature formula is challenging for hypersingular integrals. In this section, we present an alternative transparent boundary condition (TBC) containing weakly singular integrals and propose a high-order numerical quadrature rule for these specific cases. 

\subsection{An alternative transparent boundary condition}

We investigate the transparent boundary condition \eqref{TBC-TM} in TM polarization and explore the application of numerical quadrature techniques to handle the weakly singular integrals involved in the transparent boundary conditions.

\begin{lemma}
The transparent boundary condition \eqref{TBC-TM} can be reformulated as follows:
\begin{align}\label{NewDtN}
\partial_{x_2}u(x_1, 0)
&=\frac{{\rm i}\kappa_0^2}{2}\int_{\Gamma} u(y_1, 0)H_0^{(1)}(\kappa_0|x_1-y_1|){\rm d}y_1\notag\\
&\quad+\frac{\rm i}{2}\partial_{x_1}\left[\int_{\Gamma}\partial_{y_1} u(y_1, 0) H_0^{(1)}(\kappa_0 |x_1-y_1|){\rm d}y_1
\right]-2{\rm i}\beta e^{{\rm i}\alpha x_1}.
\end{align}
\end{lemma}

\begin{proof}
Let $G_{\rm TM}(\boldsymbol{x}, \boldsymbol{y})$ represent the Dirichlet Green's function in the upper half-plane for TM polarization. Specifically, it is given by 
\[
	G_{\rm TM}(\boldsymbol{x}; \boldsymbol{y})=\frac{\rm i}{4}H_0^{(1)}(\kappa_0 |\boldsymbol{x}-\boldsymbol{y}|)
	-\frac{\rm i}{4}H_0^{(1)}(\kappa_0 |\boldsymbol{x}-\boldsymbol{y}'|),
\]
where $\boldsymbol{x}=(x_1, x_2)$, $\boldsymbol{y}=(y_1, y_2)$, and $\boldsymbol{y}'=(y_1, -y_2)$.
It follows from the Sommerfeld radiation condition \eqref{Sommerfeld} and Green's integral formula that for any $(x_1, x_2)\in \mathbb{R}_+^2$, the scattered field $u^s(x_1, x_2)$ can be expressed as
\begin{eqnarray*}
u^s(x_1, x_2)=-\int_{\Gamma_g\cup\Gamma} \partial_{y_2} u^s(y_1, 0) G_{\rm TM}(x_1, x_2; y_1, 0){\rm d} y_1\\
   		+\int_{\Gamma_g\cup\Gamma}u^s(y_1, 0) \partial_{y_2} G_{\rm TM}(x_1, x_2; y_1, 0){\rm d} y_1.
\end{eqnarray*}

Based on the homogeneous Dirichlet boundary condition \eqref{pec-pec-TM}, the total field satisfies
\begin{align}\label{InteS}
u(x_1, x_2)-u^b(x_1, x_2)
&=-\int_{\Gamma_g\cup\Gamma} \Big[\partial_{y_2} (u(y_1, 0)-u^b(y_1, 0)) G_{\rm TM}(x_1, x_2; y_1, 0)\notag\\
&\quad -(u(y_1, 0)-u^b(y_1, 0)) \partial_{y_2} G_{\rm TM}(x_1, x_2; y_1, 0)\Big]{\rm d} y_1\notag\\
&=\int_{\Gamma} u(y_1, 0) \partial_{y_2} G_{\rm TM}(x_1, x_2; y_1, 0){\rm d}y_1.
\end{align}
where $u^b=u^i-u^r=e^{{\rm i}(\alpha x_1-\beta x_2)}-e^{{\rm i}(\alpha x_1+\beta x_2)}$ is referred to the background field.

Define an auxiliary function as the Dirichlet Green's function in the upper half-plane for TE polarization:
\[
	G_{\rm TE}(\boldsymbol{x}, \boldsymbol{y})=\frac{\rm i}{4}H_0^{(1)}(\kappa_0 |\boldsymbol{x}-\boldsymbol{y}|)
	+\frac{\rm i}{4}H_0^{(1)}(\kappa_0 |\boldsymbol{x}-\boldsymbol{y}'|).
\]
Noting $\partial_{y_2}G_{\rm TM}(x_1, x_2; y_1, y_2)=-\partial_{x_2}G_{\rm TE}(x_1, x_2; y_1, y_2)$, we have from 
\eqref{InteS} that 
\begin{equation}\label{InteS2}
u(x_1, x_2)-u^b(x_1, x_2)=-\int_{\Gamma} u(y_1, 0)\partial_{x_2}G_{\rm TE}(x_1, x_2; y_1, 0){\rm d}y_1.
\end{equation}

Taking the partial derivative of \eqref{InteS2} with respect to $x_2$, we obtain 
\begin{align}\label{KeyI2}
&\partial_{x_2}u(x_1, x_2)-\partial_{x_2}u^b(x_1, x_2)  =
 -\int_{\Gamma} u(y_1, 0)\partial^2_{x_2}G_{\rm TE}(x_1, x_2; y_1, 0){\rm d}y_1\notag\\
&=\kappa_0^2\int_{\Gamma} u(y_1, 0)G_{\rm TE}(x_1, x_2; y_1, 0){\rm d}y_1\notag\\
&\quad -\frac{\rm i}{4}\int_{\Gamma} u(y_1, 0)\left[\partial_{x_1y_1}^2 H_0^{(1)}(\kappa_0 |\boldsymbol{x}-\boldsymbol{y}|)
+\partial_{x_1y_1}^2 H_0^{(1)}(\kappa_0 |\boldsymbol{x}-\boldsymbol{y}'|)\right]{\rm d}y_1\notag\\
&=\kappa_0^2\int_{\Gamma} u(y_1, 0)G_{\rm TE}(x_1, x_2; y_1, 0){\rm d}y_1\notag\\
&\quad +\frac{\rm i}{4}\int_{\Gamma} \partial_{y_1}u(y_1,0)
\left[\partial_{x_1} H_0^{(1)}(\kappa_0 |\boldsymbol{x}-\boldsymbol{y}|)
+\partial_{x_1} H_0^{(1)}(\kappa_0 |\boldsymbol{x}-\boldsymbol{y}'|)\right]{\rm d}y_1,
\end{align} 
where we have employed integration by parts and utilized the following facts:
\begin{align*}
\partial_{x_2}^2 G_{\rm TE}(\boldsymbol{x}; \boldsymbol{y})&=-\partial_{x_1}^2 G_{\rm TE}(\boldsymbol{x}; \boldsymbol{y})-\kappa_0^2 G_{\rm TE}(\boldsymbol{x}; \boldsymbol{y}),\\
\partial_{x_1}^2 G_{\rm TE}(\boldsymbol{x}; \boldsymbol{y}) &=
\frac{\rm i}{4}\left[\partial_{x_1}^2 H_0^{(1)}(\kappa_0 |\boldsymbol{x}-\boldsymbol{y}|)
+\partial_{x_1}^2 H_0^{(1)}(\kappa_0 |\boldsymbol{x}-\boldsymbol{y}'|)\right]\\
&=-\frac{\rm i}{4}\left[\partial_{x_1y_1}^2 H_0^{(1)}(\kappa_0 |\boldsymbol{x}-\boldsymbol{y}|)
+\partial_{x_1y_1}^2 H_0^{(1)}(\kappa_0 |\boldsymbol{x}-\boldsymbol{y}'|)\right].
\end{align*}
Taking $x_2\rightarrow 0^+$ in \eqref{KeyI2}, and using the continuity of the total field and
the single layer potential across the aperture $\Gamma$, we deduce 
\begin{align*}
&\partial_{x_2}u(x_1, 0)-\partial_{x_2}u^b(x_1, 0)\\
 &=\kappa_0^2\int_{\Gamma} u(y_1, 0) G_{\rm TE}(x_1, 0; y_1, 0){\rm d}y_1
+\frac{\rm i}{2}\int_{\Gamma}\partial_{y_1} u(y_1, 0)\partial_{x_1} H_0^{(1)}(\kappa_0 |x_1-y_1|)
{\rm d}y_1\\
&=\frac{{\rm i}\kappa_0^2}{2}\int_{\Gamma} u(y_1, 0)H_0^{(1)}(\kappa_0|x_1-y_1|){\rm d}y_1
+\frac{\rm i}{2}\partial_{x_1}\Big[\int_{\Gamma}\partial_{y_1} u(y_1, 0) H_0^{(1)}(\kappa_0 |x_1-y_1|)
{\rm d}y_1\Big],
\end{align*}
which completes the proof by noting $\partial_{x_2}u^b(x_1, 0)=-2{\rm i}\beta e^{{\rm i}\alpha x_1}$. 
\end{proof}

Substituting the expansion \eqref{fse} into  \eqref{NewDtN}, we get
\begin{align}\label{SingleNewDtNExp}
\partial_{x_2}u(x_1, 0)
&=\frac{{\rm i}\kappa_0^2}{2}\int_{0}^{w}\bigg( \sum\limits_{n=1}^N u^{(n)}\sin\frac{ n \pi y_1}{w}H_0^{(1)}(\kappa_0|x_1-y_1-a|)\bigg){\rm d}y_1\notag\\
&\quad+\frac{\rm i}{2}\partial_{x_1}\bigg[\int_{0}^{w}
\sum\limits_{n=1}^N u^{(n)}\frac{n\pi}{w}\cos\frac{n \pi y_1}{w} H_0^{(1)}(\kappa_0 |x_1-y_1-a|){\rm d}y_1\bigg]-2{\rm i}\beta e^{{\rm i}\alpha x_1}.
\end{align}
By multiplying both sides of \eqref{SingleNewDtNExp} by $\sin \frac{m \pi (x_1-a)}{w}$ and integrating over $\Gamma$, we obtain the equivalent entries of the matrix $M_{\rm TM}$ in \eqref{TM-SingleDtNOld}:
\begin{align}\label{NewDtNK-Single}
M_{\rm TM}(m, n)&=\frac{{\rm i}\kappa_0^2}{2}\int_0^{w}\left[\int_{0}^{w}\sin\frac{n\pi y}{w}H_0^{(1)}\left(\kappa_0|x-y|\right){\rm d}y\right]\sin\frac{m\pi  x}{w} {\rm d}x\notag\\ 
 &\quad -\frac{{\rm i}mn\pi^2}{2w^2}\int_0^{w}
\left[\int_{0}^{w} \cos\frac{n\pi  y}{w}  H_0^{(1)}(\kappa_0 |x-y|){\rm d}y
\right]\cos\frac{m\pi  x}{w} {\rm d}x.
\end{align}

For multiple multi-layered cavities, following the same discussion, we can derive the TBC on the $k$-th cavity as
\begin{align}\label{MulNewDtN}
\partial_{x_2}u(x_1, 0; k)
&=\frac{{\rm i}\kappa_0^2}{2}\sum\limits_{j=1}^K\int_{\Gamma_j} u(y_1, 0; j)
	H_0^{(1)}(\kappa_0|x_1-y_1|){\rm d}y_1\notag\\
&\quad+\frac{\rm i}{2}\sum\limits_{j=1}^K\partial_{x_1}\left[\int_{\Gamma_j}\partial_{y_1} u(y_1, 0; j) 
	H_0^{(1)}(\kappa_0 |x_1-y_1|){\rm d}y_1\right]
	-2{\rm i}\beta e^{{\rm i}\alpha x_1}.
\end{align}
Substituting \eqref{MulMulTM-u} into \eqref{MulNewDtN} and using a change of variables, we deduce
\begin{align}\label{MulNewDtNExp}
&\partial_{x_2}u(x_1, 0; k)
=\frac{{\rm i}\kappa_0^2}{2}\sum\limits_{j=1}^K\int_{0}^{w_j}\left( \sum\limits_{n=1}^N
	u^{(n)}_{0, j}\sin\frac{n\pi y_1}{w_j}H_0^{(1)}(\kappa_0|x_1-y_1-a_j|)\right){\rm d}y_1\notag\\
&+\frac{\rm i}{2}\sum\limits_{j=1}^K\partial_{x_1}\left[\int_{0}^{w_j}
\sum\limits_{n=1}^N u^{(n)}_{0, j}\frac{n\pi}{w_j}\cos\frac{n\pi  y_1}{w_j} 
	H_0^{(1)}(\kappa_0 |x_1-y_1-a_j|){\rm d}y_1\right]-2{\rm i}\beta e^{{\rm i}\alpha x_1}.
\end{align}
By multiplying both sides of \eqref{MulNewDtNExp} by $\sin \frac{\pi m (x_1-a_k)}{w_k}$ and integrating over $\Gamma_k$, we can obtain a new form of the matrix ${\mathbb M}_{\rm TM}$ in \eqref{MultiTM}, which only contains weakly singular integrals and has entries respresented as follows: 
\begin{align*}
M_{k, j}(m, n)&=\frac{{\rm i}\kappa_0^2}{2}\int_{0}^{w_k}
	\int_{0}^{w_j}\sin\frac{n\pi y}{w_j} 
	H_0^{(1)}\left(\kappa_0|x+a_k-y-a_j|\right)\sin\frac{m\pi x}{w_k} {\rm d}y{\rm d}x\notag\\
&\quad -\frac{{\rm i}mn\pi^2}{2 w_j w_k}
 	\int_{0}^{w_k}\int_{0}^{w_j}
 	\cos\frac{n\pi y}{w_j} H_0^{(1)}\left(\kappa_0 |x+a_k-y-a_j|\right)
 	\cos\frac{m\pi x}{w_k}{\rm d}y {\rm d}x,
\end{align*}
where $k, j=1, \dots, K$, and $m, n=1, \dots, N$.  

\subsection{Numerical quadratures}

The matrices ${\mathbb M}_{\rm TM}$ and ${\mathbb M}_{\rm TE}$  contain weakly singular integrals in the following forms subsequent to the use of a change of variables:
\begin{equation}\label{KeyIntegration}
\begin{aligned}
\int_0^{2\pi}\int_0^{2\pi} \sin\frac{ns}{2}
	H_0^{(1)}\left(\frac{\kappa_0 w}{2\pi}|s-t|\right)\sin\frac{mt}{2}{\rm d}s{\rm d}t,\\
	 \int_0^{2\pi}\int_0^{2\pi} \cos\frac{ns}{2} 
	H_0^{(1)}\left(\frac{\kappa_0 w}{2\pi}|s-t|\right)\cos\frac{mt}{2}{\rm d}s{\rm d}t.
\end{aligned}
\end{equation}
In this section, we develop an efficient approach to accurately evaluate \eqref{KeyIntegration}.

\begin{lemma}\label{lemma-mn}
If $m+n$ is odd, then
\begin{align}
\label{lemma-mn-1}\int_0^{2\pi}\left(\int_0^{2\pi} H_0^{(1)}\left(\frac{\kappa_0 w}{2\pi}|s-t|\right)
	\sin\frac{ns}{2}{\rm d}s\right)\sin\frac{mt}{2}{\rm d}t &=0,\\
\label{lemma-mn-2}\int_0^{2\pi}\left(\int_0^{2\pi} H_0^{(1)}\left(\frac{\kappa_0 w}{2\pi}|s-t|\right)
	\cos\frac{ns}{2}{\rm d}s\right)\cos\frac{mt}{2}{\rm d}t &=0.
\end{align}
\end{lemma}

\begin{proof}
We only show the proof of \eqref{lemma-mn-1} since the proof is similar to \eqref{lemma-mn-2}. For any $t\in[0, 2\pi]$, define 
\[
	F(t)=\int_0^{2\pi} H_0^{(1)}\left(\frac{\kappa_0 w}{2\pi}|s-t|\right)\sin\frac{ns}{2}{\rm d}s.
\]
Let $\xi=2\pi-s$, we have from a simple calculation that 
\begin{align*}
	F(2\pi-t)&=\int_0^{2\pi} H_0^{(1)}\left(\frac{\kappa_0 w}{2\pi}|s-2\pi+t|\right)\sin\frac{ns}{2}{\rm d}s \\
		&=\int_0^{2\pi} H_0^{(1)}\left(\frac{\kappa_0 w}{2\pi}|2\pi-\xi-2\pi+t|\right)\sin\frac{n}{2}(2\pi-\xi){\rm d}\xi=(-1)^{(n+1)} F(t).
\end{align*}
A straightforward calculation shows that 
\begin{align*}
&\int_0^{2\pi}\left(\int_0^{2\pi} H_0^{(1)}
	\left(\frac{\kappa_0 w}{2\pi}|s-t|\right)\sin\frac{ns}{2}{\rm d}s\right)\sin\frac{mt}{2}{\rm d}t\\
&=\int_0^{\pi} F(t)\sin\frac{mt}{2}{\rm d}t+\int_{\pi}^{2\pi} F(t)\sin\frac{mt}{2}{\rm d}t\\
&=\int_0^{\pi} F(t)\sin\frac{mt}{2}{\rm d}t+(-1)^{(m+n)}\int_{0}^{\pi} F(t)\sin\frac{mt}{2}{\rm d}t=0, 
\end{align*}
if $m+n$ is odd, which completes the proof. 
\end{proof}

Next, we address numerical quadrature in scenarios where the sum of $m$ and $n$ results in an even number. The established technique outlined in \cite[Section 3.5]{coltonbook13} for computing weakly singular integrals associated with the zeroth-order Hankel function necessitates the integrand function to be $2\pi$-periodic. Regrettably, this condition cannot be met in our current context.

It is clear to note from \eqref{KeyIntegration} that 
\begin{align}\label{SingluarIntegrationH0}
&\int_0^{2\pi}\int_0^{2\pi} \sin\frac{ns}{2}
	H_0^{(1)}\Big(\frac{\kappa_0 w}{2\pi}|s-t|\Big)\sin\frac{mt}{2}{\rm d}s{\rm d}t\notag\\
&	=\int_0^{2\pi}\int_0^{2\pi} \sin\frac{ns}{2}
	\left[H_0^{(1)}\Big(\frac{\kappa_0 w}{2\pi}|s-t|\Big)
		-\frac{2{\rm i}}{\pi}J_0\Big(\frac{\kappa_0 w}{2\pi}|s-t|\Big)\ln |s-t|\right]\sin\frac{mt}{2}{\rm d}s{\rm d}t\notag\\
&\quad+\frac{2{\rm i}}{\pi}\int_0^{2\pi}\int_0^{2\pi}
	J_0\Big(\frac{\kappa_0 w}{2\pi}|s-t|\Big)\ln |s-t| \sin\frac{ns}{2}\sin\frac{mt}{2}{\rm d}s{\rm d}t. 
\end{align}
Recalling the power series of the Bessel function (cf. \cite{OLBC-NY-2010})
\begin{align*}
H_0^{(1)}(z)&=J_0(z)+{\rm i} Y_0(z)\\
&= J_0(z)+\frac{2{\rm i}}{\pi}\Big(\ln\frac{z}{2}+\gamma\Big)J_0(z)+\frac{2{\rm i}}{\pi}
\sum\limits_{k=1}^\infty (-1)^{k+1}\bigg(\sum\limits_{j=1}^k \frac{1}{j}\bigg)\frac{\left(\frac{1}{4}z^2\right)^k}{(k!)^2}, 
\end{align*}
we know that the first term of \eqref{SingluarIntegrationH0} is analytic. Moreover, we have 
\[
	\lim\limits_{t\rightarrow s}H_0^{(1)}\Big(\frac{\kappa_0 w}{2\pi}|s-t|\Big)
		-\frac{2{\rm i}}{\pi}J_0\Big(\frac{\kappa_0 w}{2\pi}|s-t|\Big)\ln |s-t|
		=1+\frac{2{\rm i}}{\pi}\gamma+\frac{2{\rm i}}{\pi}\ln \frac{\kappa_0 w}{4\pi}.
\]
The integrand function of the second term is weakly singular, and directly applying a high order quadrature rule is not feasible. Our idea is to repeatedly apply integration by parts to increase the regularity of the integrand function.

Specifically, it follows from the expansion of Bessel function (cf. \cite{OLBC-NY-2010}) that 
\begin{align}\label{DecomposeJ0}
	J_0\Big(\frac{\kappa_0 w}{2\pi}|s-t|\Big)=\sum\limits_{k=0}^\infty(-1)^k\frac{1}{(k!)^2}\Big[\frac{\kappa_0 w}{4\pi}(s-t)\Big]^{2k}
	=\sum\limits_{k=0}^K (-1)^k\frac{1}{(k!)^2}\Big[\frac{\kappa_0 w}{4\pi}(s-t)\Big]^{2k}\notag\\
	+\bigg[J_0\Big(\frac{\kappa_0 w}{2\pi}|s-t|\Big)
	-\sum\limits_{k=0}^K (-1)^k\frac{1}{(k!)^2}\Big(\frac{\kappa_0 w}{4\pi}(s-t)\Big)^{2k}\bigg].
\end{align}
Since the second term has higher regularity of $C^{2K+2}$, substituting \eqref{DecomposeJ0} into \eqref{SingluarIntegrationH0}, we only need to handle the integral
\[
S_k(n, m)=\int_0^{2\pi}\int_0^{2\pi}(t-s)^{k-1}\ln|t-s| \sin\frac{mt}{2}\sin\frac{ns}{2}{\rm d}s{\rm d}t. 
\]

Through straightforward calculations and integration by parts, we deduce
\begin{align}\label{RecursiveSin}
S_k(n, m)=-\frac{1}{k(k+1)}\left(\frac{m}{2}\right)^2 S_{k+2}(n, m)+\frac{1}{k(k+1)^2}\left(\frac{m}{2}\right)^2 T^{(1)}_{k}(n, m)\notag\\
+\frac{m}{2k^2}T_k^{(2)}(n, m) -\frac{m}{2k(k+1)^2}\left[(-1)^{m+n}-(-1)^k\right]T_k^{(3)}(n)\notag\\
+\frac{m}{2k(k+1)}\left[(-1)^{m+n}-(-1)^k\right]W_k(n), 
\end{align}
where
\begin{align*}
&T^{(1)}_{k}(n, m)=\int_0^{2\pi}\int_0^{2\pi}\sin\frac{ns}{2}(t-s)^{k+1}\sin\frac{mt}{2}{\rm d}s{\rm d}t,\\
&T_k^{(2)}(n, m)=\int_0^{2\pi}\int_0^{2\pi}\sin\frac{ns}{2}(t-s)^{k}\cos\frac{mt}{2}{\rm d}s{\rm d}t
,\\
&T_k^{(3)}(n)=\int_0^{2\pi}\sin\frac{ns}{2}s^{k+1}{\rm d}s,\quad
W_k(n)=\int_0^{2\pi}\sin\frac{ns}{2}s^{k+1}\ln s\,{\rm d}s.
\end{align*}
The regularity of $W_k$ can be increased by applying the same idea. Explicitly, we have 
\begin{eqnarray*}
	W_k(n)=A_k(n)-\left(\frac{n}{2}\right)^2\frac{1}{(k+2)(k+3)}W_{k+2}(n),
\end{eqnarray*}
where
\begin{align*}
A_k(n)=\frac{n}{2}\frac{1}{k+2}\left(\frac{1}{k+2}+\frac{1}{k+3}\right)\int_{0}^{2\pi} s^{k+2}\cos\frac{ns}{2}\\
-\frac{(-1)^n n}{2(k+2)(k+3)}\left(2\pi\right)^{k+3}\ln(2\pi).
\end{align*}

The corresponding parts of the integrals involving cosine functions can be evaluated similarly as follows:
\begin{align}\label{RecursiveCos}
	P_k(n, m)&=\int_0^{2\pi}\int_0^{2\pi}(t-s)^{k-1}\ln|t-s| \cos\frac{mt}{2}\cos\frac{ns}{2}{\rm d}s{\rm d}t\notag\\
	&=-\left(\frac{m}{2}\right)^2\frac{1}{k(k+1)}P_{k+2}(n, m)+
	\left(\frac{m}{2}\right)^2\frac{1}{k(k+1)^2}U_k^{(1)}(n, m)
	-\frac{m}{2k^2}U_k^{(2)}(n, m)\notag\\
	&\quad -\left[(-1)^{(m+n)}-(-1)^k\right]\frac{1}{k^2}U_k^{(3)}(n)
	 	+\left[(-1)^{(m+n)}-(-1)^k\right]\frac{1}{k}X_k(n),
\end{align}
where
\begin{align*}
&U_k^{(1)}(n, m)=\int_0^{2\pi}\int_0^{2\pi}\cos\frac{ns}{2}(t-s)^{k+1}\cos\frac{mt}{2}{\rm d}s{\rm d}t,\\
&U_k^{(2)}(n, m)=\int_0^{2\pi}\int_0^{2\pi}\cos\frac{ns}{2}(t-s)^{k}\sin\frac{mt}{2}{\rm d}s{\rm d}t,\\
&U_k^{(3)}(n)=\int_0^{2\pi}\cos\frac{ns}{2}s^{k}{\rm d}s,\quad
X_k(n)=\int_0^{2\pi}\cos\frac{ns}{2}s^{k}\ln s{\rm d}s.
\end{align*}
Moreover, the function $X_k$ satisfies the recurrence relation
\[
X_k(n)=Y_k(n)-\left(\frac{n}{2}\right)^2\frac{1}{(k+1)(k+2)}X_{k+2}(n),
\]
where
\begin{align*}
Y_k(n)=-\frac{n}{2}\frac{1}{k+1}\left(\frac{1}{k+1}+\frac{1}{k+2}\right)\int_0^{2\pi}s^{k+1}\sin\frac{ns}{2}{\rm d}s\\
	+\frac{(-1)^n (2\pi)^{k+1}}{\left(k+1\right)^2}\left[(k+1)\ln(2\pi)-1\right].
\end{align*}

By using the recurrence relations \eqref{RecursiveSin} and \eqref{RecursiveCos}, we can iteratively enhance the regularity of the integrand functions in \eqref{SingluarIntegrationH0} until they meet the requirements of integral quadrature formulas, such as the Gaussian quadrature formula.

\section{Numerical Experiments}\label{Section:NE}

In this section, we present a series of numerical examples to provide compelling evidence regarding the efficacy of the method in precisely analyzing electromagnetic scattering phenomena linked to  rectangular cavities.

 \subsection{Order of accuracy}\label{Example1}
 
 First, we assess the accuracy of the proposed method, demonstrating its precision and reliability in handling the cavity scattering problems in both TM and TE polarizations. In this experiment, we aim to evaluate the convergent order of the proposed method in the $L^2$-norm. Due to the unavailability of an analytic solution, we conduct a comparison of the results obtained on a finer mesh.
 
 The cavity is situated in the interval $[-0.5, 0.5]$ with a depth of $h=1.5$. It is subjected to illumination by a plane wave with a wavenumber of $\kappa_0=1.5$ and an incidence angle of $\theta=\pi/9$. For the numerical evaluation of \eqref{KeyIntegration}, we opt for the composite 4-point Gaussian quadrature formula as a representative numerical integration technique. We set the truncation number to $N=30$. It is worth mentioning that by employing higher-order numerical quadrature formulas, even higher orders of convergence can be achieved. Figure \ref{ex1:order} displays the results for both TM and TE polarizations. The $x$-axis represents the mesh size, while the $y$-axis corresponds to the error measured in the $L^2$-norm. The red dashed line depicts the convergent order of the proposed method. For comparison purposes, we also include the solid blue curve, which represents the theoretical eighth-order convergence. The results clearly demonstrate that the proposed method exhibits an accuracy with an order of $O(h^8)$, thereby affirming its effectiveness and reliability in numerical simulations.

\begin{figure}[h]
\centering
\includegraphics[width=0.45\textwidth]{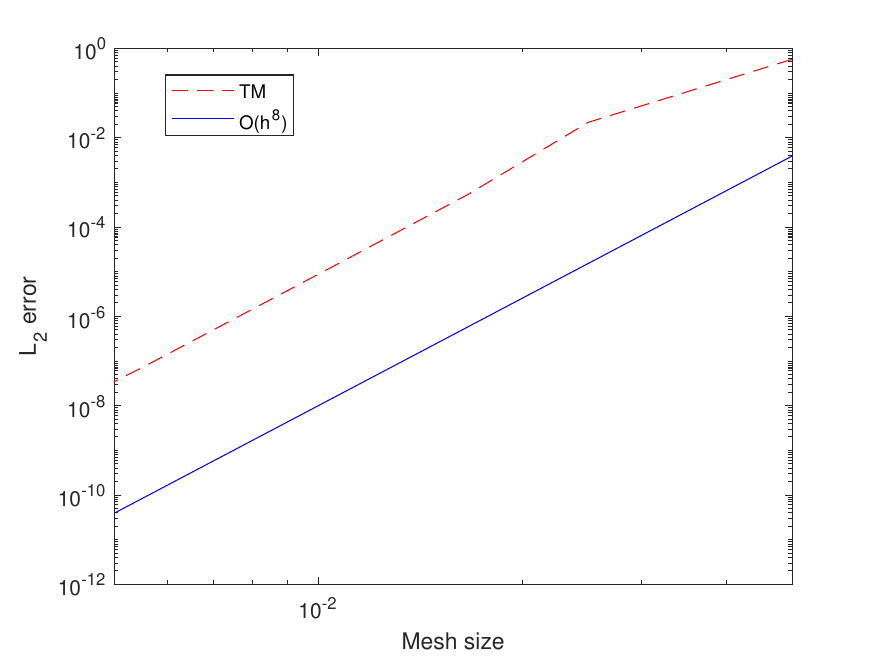}
\includegraphics[width=0.45\textwidth]{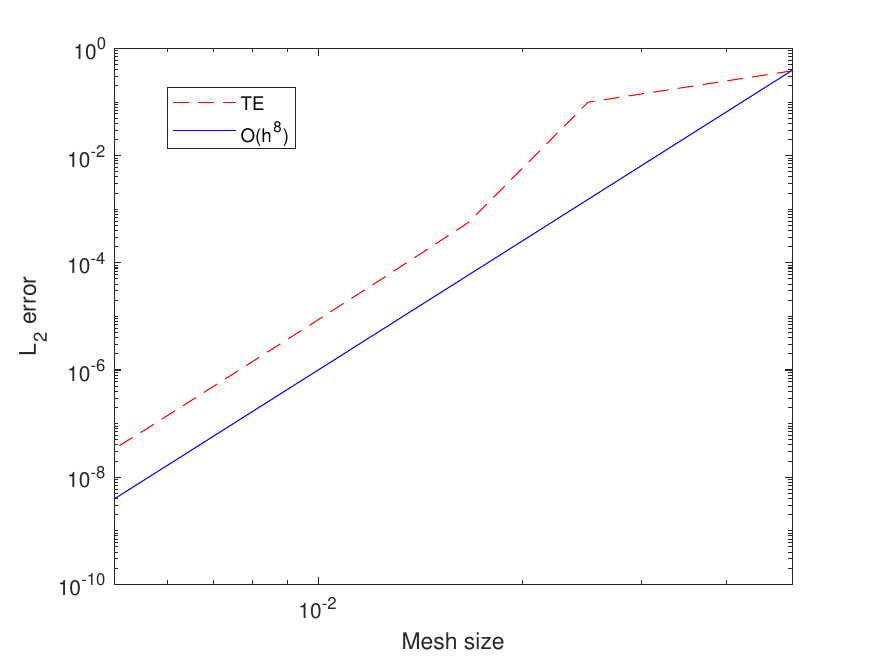}
\caption{Example \ref{Example1}: (Left) Order of convergence in TM polarization; (Right) Ordr of convergence in
TE polarization.}
\label{ex1:order}
\end{figure}

\subsection{Radar cross-section}\label{Example2}

Next, we explore the application of our proposed method in radar cross-section (RCS) computations, which highlights its capability to accurately predict the scattering behavior of practical radar systems.


In two dimensions, the RCS is defined by (cf. \cite{Jin-2002})
\[
\sigma(\varphi):=\lim\limits_{r=\rightarrow\infty}2\pi r\frac{|u^s(x, y;
\varphi)|^2}{|u^i(x, y; \theta)|^2},
\]
where $u^i$ is the incident field, $u^s$ denotes the scattered field, and $\theta$ and $\varphi$ represent the incident and observation angles, respectively. When $\theta$ and $\varphi$ are equal, $\sigma$ is referred to as the backscatter RCS, and it is defined as
\[
 \text{Backscatter RCS}(\sigma)(\varphi)=10\log_{10}\sigma(\varphi) {\rm dB}. 
\]
In TM polarization, with measurements taken on the aperture of the cavity, the RCS can be expressed as follows (cf. \cite{YBL-CSIAM-2020}):
 \begin{equation*}\label{RCS_TM}
\sigma(\varphi) =\kappa_0 \Big| \sin\varphi\int_{\Gamma} u^s(x, 0)
e^{{\rm i}\kappa_0\cos\varphi x}{\rm d}x\Big|^2.
\end{equation*}

In this experiment, we replicate a benchmark example documented in \cite{Jin-2002}. 
Specifically, we consider the backscatter RCS of a single rectangular cavity in TM polarization. The cavity has a width $w=\lambda$ and a depth $h=\lambda/4$, where $\lambda=2\pi/\kappa_0$ represents the wavelength in free space. For our analysis, we set $\kappa_0=32\pi$ and $N=150$, along with an incident angle of $\theta=\pi/3$. The obtained numerical backscatter RCS results are shown in Figure \ref{ex2:RCS}. To provide a basis for comparison, we also present the results obtained using the adaptive finite element TBC method \cite{YBL-CSIAM-2020}. The adaptive TBC method is terminated when the total number of nodal points reaches 15000. In the left part of Figure \ref{ex2:RCS}, we plot the results for the empty cavity. The solid red line represents the backscatter RCS computed by our proposed method, while the blue circle points represent the results obtained from the adaptive TBC method. In the right part of Figure \ref{ex2:RCS}, we consider the cavity filled with a lossy medium characterized by an electric permittivity of $\epsilon=4+{\text{i}}$ and a magnetic permeability of $\mu=1$. Similar to the empty cavity case, we depict the results of the proposed method using a solid red line and the adaptive TBC method using blue circle points. The results clearly demonstrate the consistency between the proposed method and the adaptive TBC method. However, it is worth noting that the proposed method only requires solving a small scale system with 151 unknowns, making it significantly more efficient in accurately computing the backscatter RCS of the cavity compared to the adaptive TBC method. This highlights the efficiency and accuracy of the proposed method in tackling the backscatter RCS computation for the considered cavity configuration.
 
\begin{figure}[h]
\centering
\includegraphics[width=0.45\textwidth]{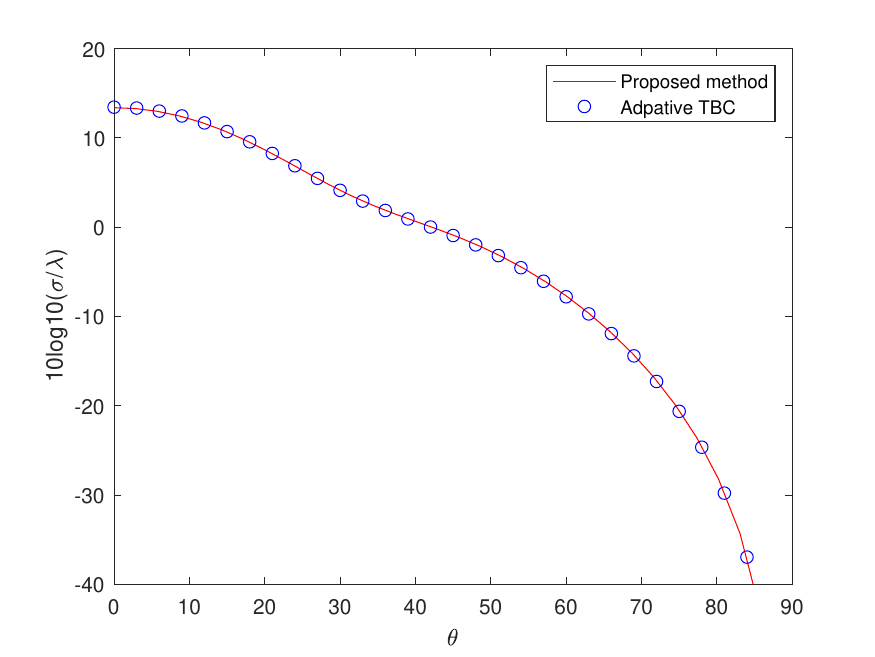}
\includegraphics[width=0.45\textwidth]{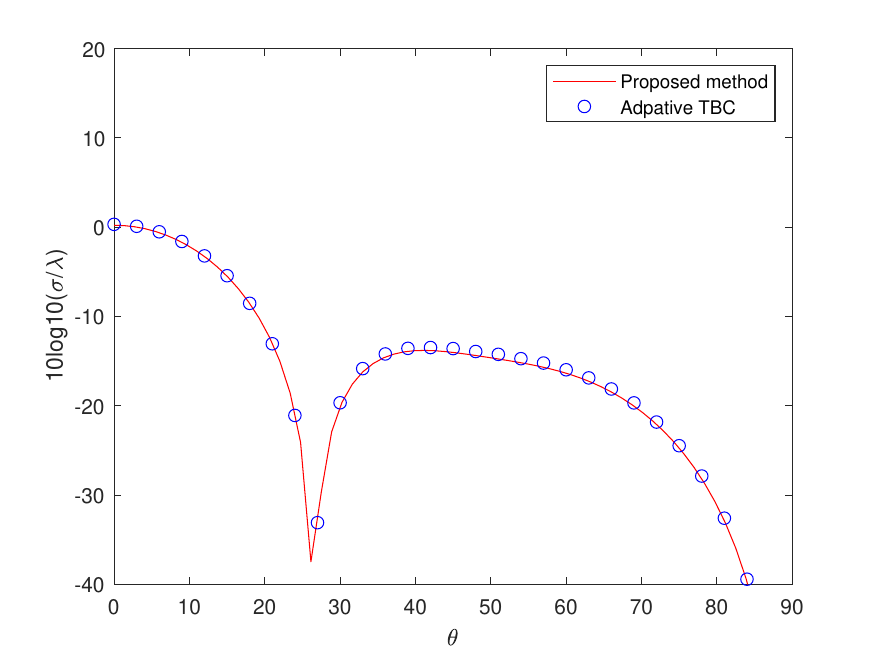}
\caption{Example \ref{Example2}: (Left) The backscatter RCS 
of an empty cavity; (Right) The backscatter RCS of a lossy cavity.}
\label{ex2:RCS}
\end{figure}

\subsection{Subwavelength enhancement}\label{Example3}

In this experiment, we explore the efficacy of our proposed method in analyzing field enhancement phenomena in subwavelength structures. By focusing on TE polarization, we can demonstrate potential applications of the method in nanophotonics and metamaterial research. 

First, we consider the case of a single cavity with the following dimensions: cavity width $w=50\,\mu m$ and cavity depth $h=1\, cm$. The cavity is illuminated from above at an incident angle of $\theta=\pi/6$. To study the behavior of the electric field enhancement, we vary the wavelength in a range from $0.2\, cm$ to $62\, cm$. Figure \ref{ex3:enhancement} presents the plot of the electric field enhancement factor $Q_E$ against the wavenumber $\kappa$, where $Q_E$ is defined as the ratio of the $L^2$-norm of the electric field $u$ to the $L^2$-norm of the incident electric field $u^i$, both integrated over the cavity domain $D$, i.e.,  
\[
Q_E=\frac{\|u\|_{L^2(D)}}{\|u^i\|_{L^2(D)}}. 
\]
As depicted in the left part of Figure \ref{ex3:enhancement}, it can be observed that the enhancement factors exhibit peaks at resonant frequencies located near $\kappa=\pi/2+n\pi$, where $n$ is an integer. These numerical findings are consistent with the theoretical results reported in \cite{GLY-Springer}, highlighting the capability of our proposed method in capturing the electric field enhancement phenomena in the subwavelength cavity structure. 

In the case of two cavities, we position them in the ground with a width of $w=0.2\,\mu m$, a depth of $h=1.5\,\mu m$, and a distance of $d=0.5\,\mu m$ between them. The two cavities are illuminated from above by a plane wave with an incident angle of $\theta=-\pi/9$. For our analysis, we vary the wavelength in the range from $500\, cm$ to $2500\, cm$. As shown in the right part of Figure \ref{ex3:enhancement}, the blue dashed curve represents the enhancement factor for a single cavity with the same width and depth, exhibiting a resonance frequency around $\kappa_0=9000$. On the other hand, the red dashed line and black solid line in the plot represent the enhancement factors for the left and right cavity, respectively. It is evident from the results that the type of enhancement for the two cavities is different. The left cavity exhibits a single peak near $\kappa_0$, while the right cavity shows two resonance frequencies around $\kappa_0$. The enhancement first peaks to the left of $\kappa_0$ and then decreases dramatically. After that, the enhancement achieves a second peak to the right of $\kappa_0$. These results demonstrate that the enhancement caused by the two cavities exhibits both an antisymmetric mode and a symmetrical mode. The numerical findings are in line with the theoretical results reported in \cite{BBT-MMS-2010} and the experimental observations conducted in \cite{PQBR-PRL-2006}. 

\begin{figure}[h]
\centering
\includegraphics[width=0.45\textwidth]{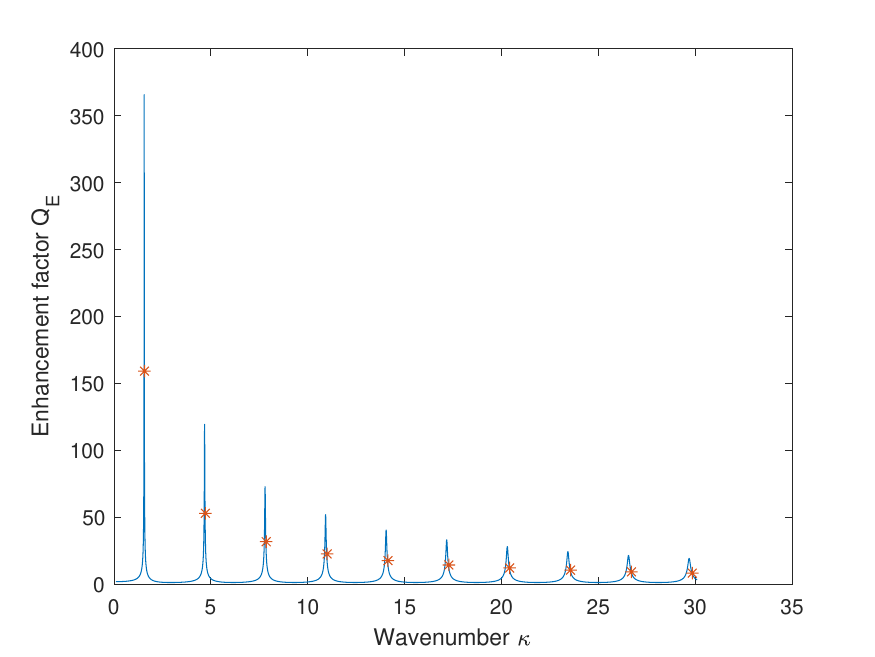}
\includegraphics[width=0.45\textwidth]{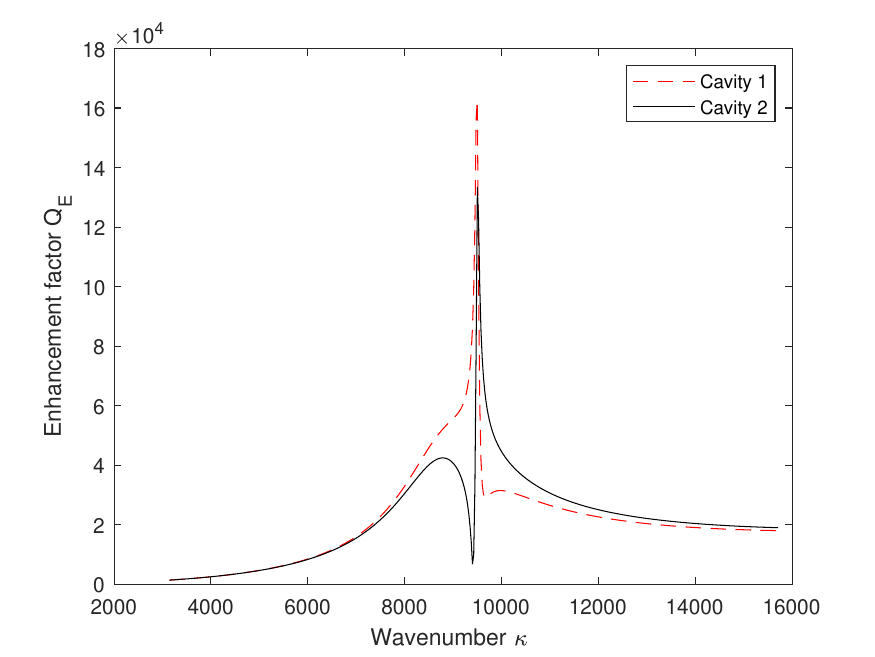}
\caption{Example \ref{Example3}: (Left) The electric field enhancement factor by a single cavity; (Right)
The electric field enhancement factor by two cavities.}
\label{ex3:enhancement}
\end{figure}
 
\subsection{Multiple multi-layered cavities}\label{Example4}

Finally, we present a challenging scenario involving the scattering by multiple cavities, which demonstrates the suitability of our proposed method in effectively handling complex configurations of multiple multi-layered cavities.

In this experiment, we consider both TM and TE polarizataions for a configuration comprising three cavities, situated at the following locations: $[-0.6, -0.1]\times [0, -0.1]\cup [0, 0.2]\times [0, -0.5] \cup [0.3, 0.6]\times[0, -0.3]$. The first cavity is assumed to be empty, while the second cavity is filled with a three-layered medium. This medium is characterized by wavenumbers $\kappa=\pi, 2\pi$, and $10\pi$, and it is separated at the inner interfaces $\left\{y=-\frac{1}{6}\right\}$ and $\left\{y=-\frac{1}{3}\right\}$. The third cavity is filled with a two-layered medium, characterized by wavenumbers $\kappa=1+0.5{\text{i}}$ and $\kappa=0.5$, and it is separated at the middle of the cavity. In order to evaluate the accuracy of our proposed method, we once again conduct a comparison with the results attained through the adaptive finite element TBC method \cite{YBL-CSIAM-2020}. The comparison is presented in Figures \ref{ex4:MultiTM} and \ref{ex4:MultiTE}. In both figures, the red solid lines and blue circle points depict the magnitudes of the total electric field on the diagonal of the left, middle, and right cavities, respectively, obtained using our proposed method and the adaptive TBC method. The comparison evidently demonstrates that the proposed method produces accurate outcomes, as they closely align with the results obtained from the adaptive TBC method. Nevertheless, it is important to note that the proposed method necessitates solving a substantially smaller system.
 
 \begin{figure}[h]
\centering
\includegraphics[width=0.32\textwidth]{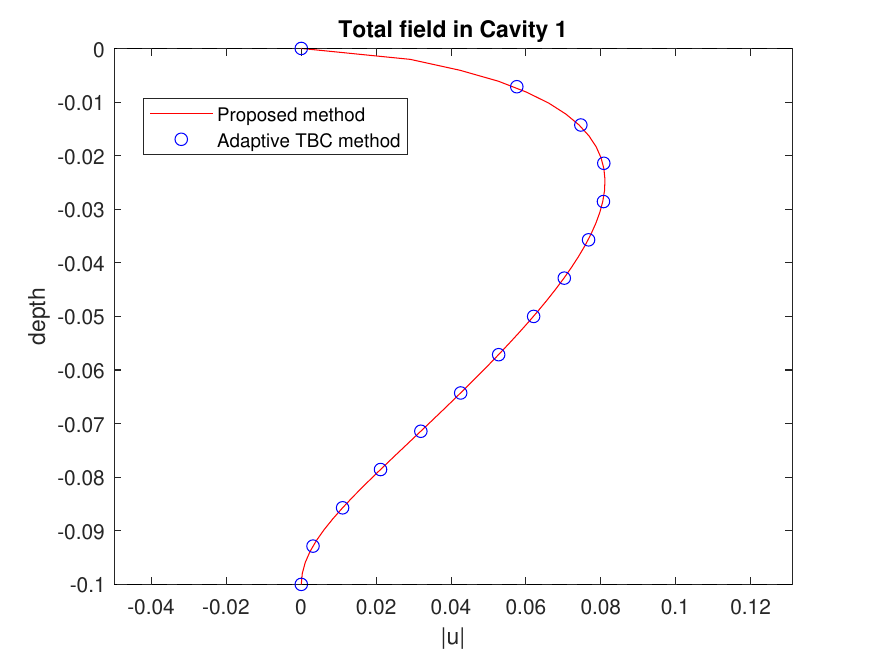}
\includegraphics[width=0.32\textwidth]{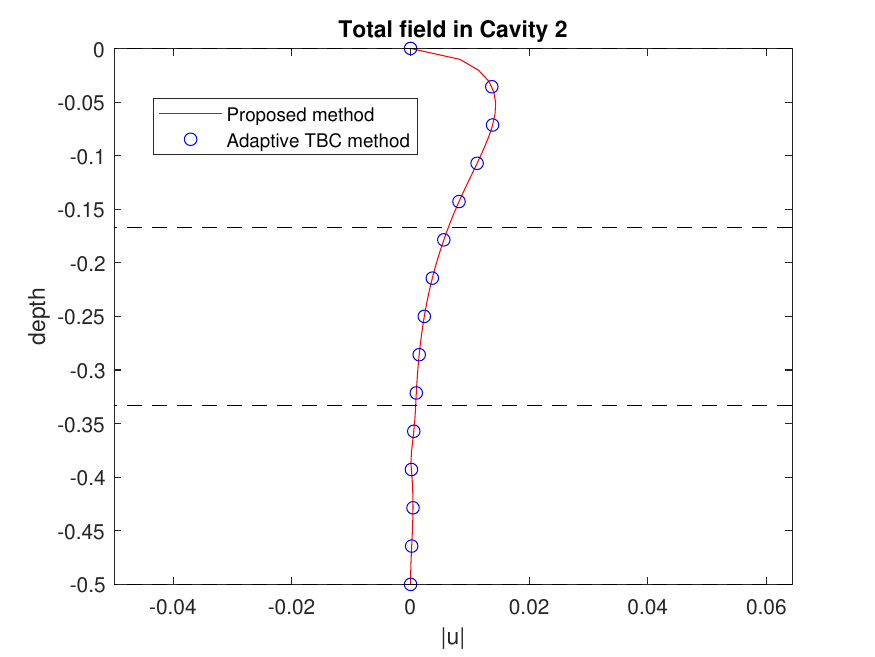}
\includegraphics[width=0.32\textwidth]{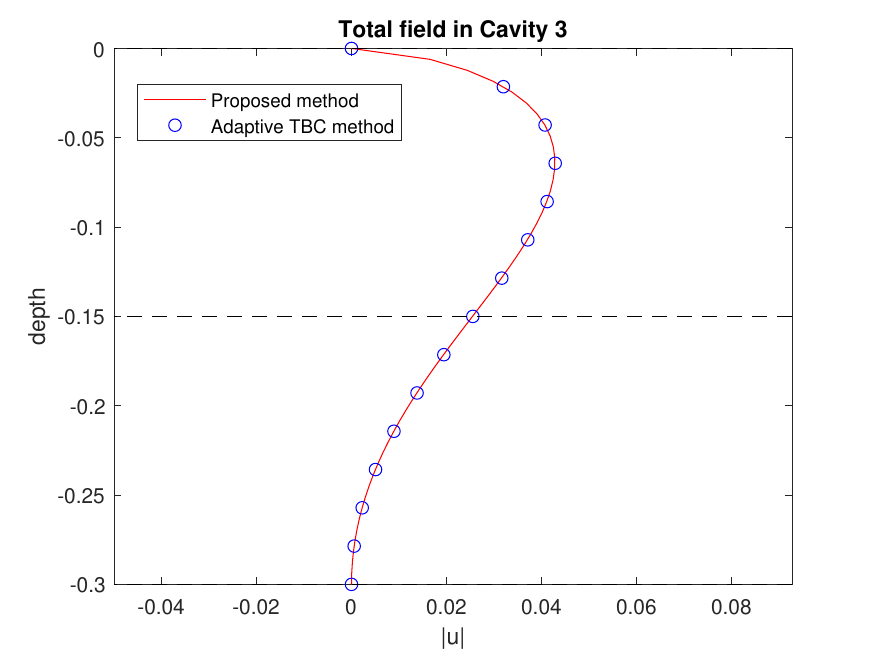}
\caption{Example \ref{Example4}: The total field on the diagonal of left, middle, and right cavity
in TM polarization.}
\label{ex4:MultiTM}
\end{figure}

\begin{figure}[h]
\centering
\includegraphics[width=0.32\textwidth]{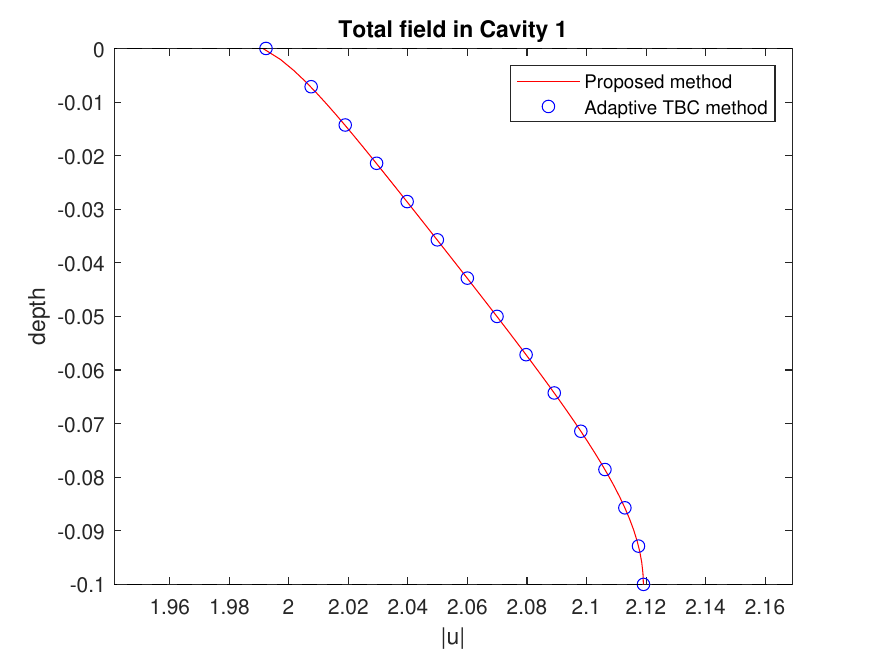}
\includegraphics[width=0.32\textwidth]{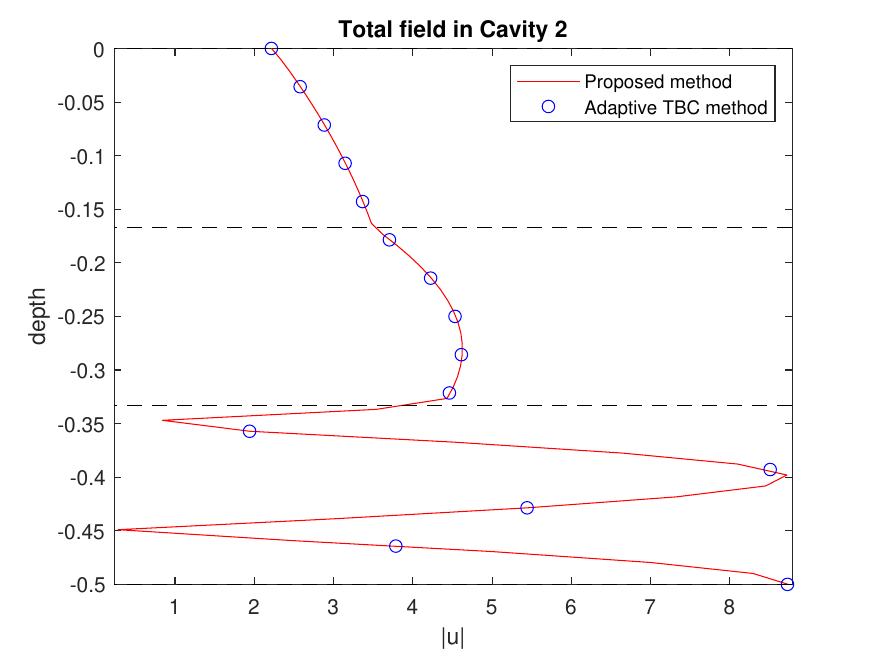}
\includegraphics[width=0.32\textwidth]{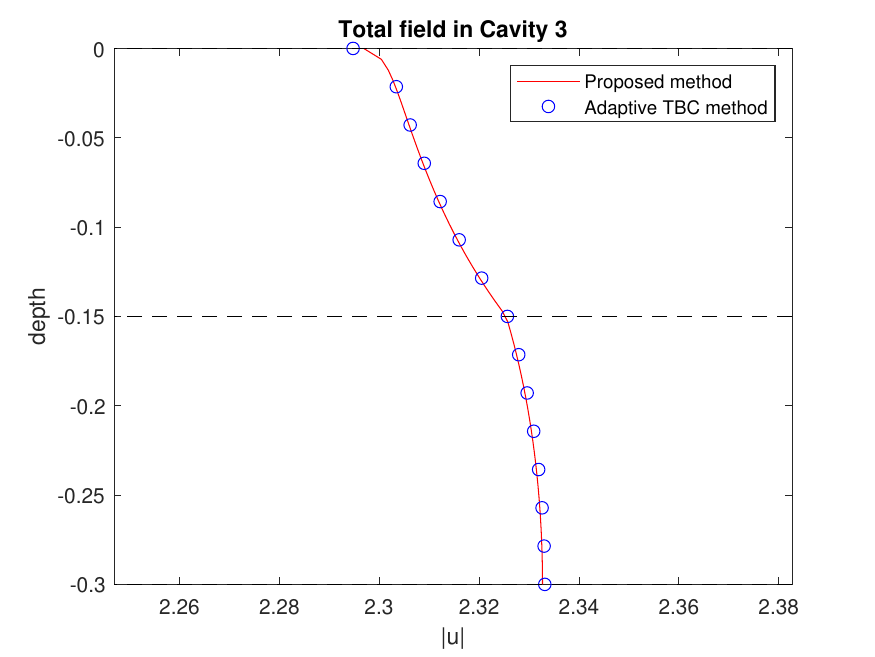}
\caption{Example \ref{Example4}: The total field on the diagonal of left, middle, and right cavity
in TE polarization.}
\label{ex4:MultiTE}
\end{figure}

\section{Conclusion}\label{Section:C}

This paper introduces a highly efficient and accurate numerical method for addressing electromagnetic scattering problems involving rectangular cavities. By using the Fourier series expansion, the original boundary value problem is transformed into one-dimensional ordinary differential equations for the Fourier coefficients. A connection formula is established to link the Fourier coefficients in each layer to those on the aperture of the cavity. This approach enables us to solve the system solely on the aperture and store the connection formula of the Fourier coefficients, significantly reducing computational resources.

Furthermore, we propose an alternative TBC on the aperture, which involves only weakly singular integrals. To handle the singularity of the Hankel function, we utilize the power series of Bessel functions to deduce a recursive formula, enhancing the smoothness of the integrand function. Consequently, high-order Gaussian quadratures can be employed, enhancing the efficiency of the method.

The numerical results demonstrate the effectiveness of our approach in accurately solving scattering problems with rectangular cavities. This method proves to be a valuable tool with diverse practical applications in radar systems, wireless communications, metamaterials, photonic devices, and more. Future research could explore extending this method to the three-dimensional Maxwell's equations and other geometries and optimizing its efficiency for even more complex scattering scenarios.

\end{document}